\documentclass{amsart}

\usepackage{hyperref}
\usepackage{amsmath,amsthm,amssymb}

\newtheorem{thm}{Theorem}
\newtheorem{prop}[thm]{Proposition}
\newtheorem{lem}[thm]{Lemma}
\newtheorem{cor}[thm]{Corollary}
\newtheorem{clm}[thm]{Claim}

\theoremstyle{definition}
\newtheorem{dfn}[thm]{Definition}
\newtheorem{rem}[thm]{Remark}
\newtheorem{ex}[thm]{Example}
\newtheorem{prob}[thm]{Problem}

\theoremstyle{remark}
\newtheorem*{org}{Organization}
\newtheorem*{ack}{Acknowledgment}

\numberwithin{thm}{section}
\numberwithin{equation}{section}

\DeclareMathOperator{\dist}{dist}
\DeclareMathOperator{\diam}{diam}
\DeclareMathOperator{\vol}{vol}
\DeclareMathOperator{\id}{id}
\DeclareMathOperator{\pr}{pr}
\DeclareMathOperator{\Cl}{Cl}

\title[Application of good coverings]{Application of good coverings to collapsing Alexandrov spaces}
\author[T. Fujioka]{Tadashi Fujioka}
\address{Department of Mathematics, Kyoto University, Kitashirakawa, Kyoto 606-8502, Japan}
\email{\href{mailto:tfujioka@math.kyoto-u.ac.jp}{tfujioka@math.kyoto-u.ac.jp}, \href{mailto:tfujioka210@gmail.com}{tfujioka210@gmail.com}}
\date{\today}
\subjclass[2010]{53C20, 53C23}
\keywords{Alexandrov spaces, collapse, good coverings, extremal subsets}
\thanks{Supported in part by JSPS KAKENHI Grant Number 15H05739}

\begin{document}

\begin{abstract}
Let $M$ be an Alexandrov space collapsing to an Alexandrov space $X$ of lower dimension.
Suppose $X$ has no proper extremal subsets and let $F$ denote a regular fiber.
We slightly improve the result of Perelman to construct an infinitely long exact sequence of homotopy groups and a spectral sequence of cohomology groups for the pair $(M,X,F)$.
The proof is an application of the good coverings of Alexandrov spaces introduced by Mitsuishi-Yamaguchi.
We also extend this result to each primitive extremal subset of $X$.
\end{abstract}

\maketitle

\section{Introduction}\label{sec:intro}

It is well-known that the class of Alexandrov spaces with a lower bound on curvature and with upper bounds on dimension and diameter is compact with respect to the Gromov-Hausdorff topology.
Thus, in order to understand the topology of spaces in this class, it is important to study the relation between the topology of a fixed Alexandrov space $X$ and that of a nearby Alexandrov space $M$.
In the noncollapsing case, i.e.\ $\dim X=\dim M$, Perelman's stability theorem (\cite{Per:alex}, cf.\ \cite{K:stab}) tells us that $M$ is homeomorphic to $X$.
On the other hand, in the collapsing case, i.e.\ $\dim X<\dim M$, there is no general theory of the structure of $M$.
The results in low dimensions (\cite{SY:3dim}, \cite{SY:vol}, \cite{Y:4dim}, \cite{MY:3dim}) suggest that $M$ has a fibration structure over $X$ in some generalized sense, where the singular fibers arise over the singular strata of $X$.
Therefore, if the singularities of $X$ are not so bad, one could expect that $M$ admits a fibration structure over $X$ in the usual sense.
Let us formulate this problem as follows.
Fix an upper bound $n$ for dimension and a lower bound $\kappa$ for curvature.

\begin{prob}\label{prob:main}
Let $X$ be a $k$-dimensional compact Alexandrov space satisfying some regularity condition, where $k<n$.
Suppose $\mu>0$ is small enough and let $M$ be an $n$-dimensional Alexandrov space that is $\mu$-close to $X$ with respect to the Gromov-Hausdorff distance.
Find a fibration structure of $M$ over $X$.
\end{prob}

There are two remarkable results for this problem.
One is Yamaguchi's fibration theorem (\cite{Y:col}), which asserts that if $X$ and $M$ are both Riemannian manifolds, then $M$ admits a structure of locally trivial fibration over $X$.
In the context of Alexandrov spaces, the assumption that $X$ is a Riemannian manifold is replaced by the one that every point of $X$ is $(k,\delta)$-strained.
Although it is expected that under this assumption the theorem holds for a general Alexandrov space $M$, the most general case is still open (see \cite{Y:conv}, \cite{RX}, \cite{X}, \cite{XY}, \cite{F:fibr}).

The other remarkable result is due to Perelman \cite{Per:col}, which concerns the case where $X$ has no proper extremal subsets.
Roughly speaking, this condition means that $X$ has no singular strata.
Hence it is much weaker than Yamaguchi's assumption and is actually optimal.
Under this assumption Perelman showed that $M$ admits a structure of Serre fibration in a certain weak sense.
In particular, he constructed a homotopy exact sequence for $M$, $X$, and a regular fiber (that is, a fiber over a regular point of $X$).
The precise statement is the following.
Let $\pi_i$ denote the $i$-th homotopy group.

\begin{thm}[\cite{Per:col}]\label{thm:per}
Suppose $X$ has no proper extremal subsets.
Let $F$ denote a regular fiber in $M$.
Fix an integer $j\ge1$.
Then, if $\mu$ is small enough, there exist isomorphisms $\pi_i(M,F)\cong\pi_i(X) $ for all $i\le j$.
In particular, we have a long exact sequence
\[\pi_j(F)\to\pi_j(M)\to\pi_j(X)\to\pi_{j-1}(F)\to\cdots.\]
\end{thm}

Note that the choice of the Gromov-Hausdorff distance $\mu$ depends on the upper bound $j$ on the dimension of homotopy groups.
In this paper, we shall remove this restriction.
Moreover, we construct a cohomology spectral sequence of the pair $(M,X,F)$.
Our main theorem is the following.
Let $H^\ast $ denote the singular cohomology group (with arbitrary coefficients).

\begin{thm}\label{thm:main}
Suppose $X$ has no proper extremal subsets.
Let $F$ denote a regular fiber in $M$.
Then, if $\mu$ is small enough, the following hold.
\begin{enumerate}
\item There exist isomorphisms $\pi_i(M,F)\cong\pi_i(X)$ for all $i$.
In particular, we have an infinitely long exact sequence
\[\cdots\to\pi_{i+1}(X)\to\pi_i(F)\to\pi_i(M)\to\pi_i(X)\to\pi_{i-1}(F)\to\cdots.\]
\item There exists a spectral sequence $\{E_r\}_{r=1}^\infty$ converging to $H^\ast(M)$ whose $E_2$ term is the \v Cech cohomology of a good cover $\mathcal U$ of $X$ with values in a locally constant presheaf $\mathcal H^\ast$ on $\mathcal U$ with group $H^\ast(F)$.
In particular, the Euler characteristics satisfy the product formula
\[\chi(M)=\chi(X)\cdot\chi(F).\]
Furthermore, if $X$ is simply-connected and $H^\ast(F)$ is a finitely generated free module, then we have $E_2=H^\ast(X)\otimes H^\ast(F)$.
\end{enumerate}
\end{thm}

Note that our result is new even when $M$ is a Riemannian manifold.

The proof of the main theorem uses the good coverings of Alexandrov spaces introduced by Mitsuishi-Yamaguchi \cite{MY:good}.
A good covering of an Alexandrov space is an open covering consisting of superlevel sets of strictly concave functions of a certain type introduced by Perelman \cite{Per:mor} (cf.\ \cite{PP:ext}, \cite{K:reg}, \cite{K:rest}, \cite{Pet:semi}).
It is of course a good cover in the topological sense and has more geometric properties.
Using such coverings, Mitsuishi-Yamaguchi \cite{MY:lip} showed the Lipschitz homotopy finiteness of the class of noncollapsing Alexandrov spaces.
In this paper, we study the relation between good coverings and collapsing.

\begin{rem}\label{rem:pet1}
Petrunin \cite{Pet:semi} announced an alternative proof of Perelman's theorem using gradient flows of semiconcave functions, but the details have never been published as far as the author knows (see \cite[4.2.3, 2.3.4]{Pet:semi}).
\end{rem}

The main theorem can be extended to each primitive extremal subset $E$ of $X$.
Let $\hat E$ denote a lift of $E$ in $M$ (that is, a subset corresponding to $E$ via the Gromov-Hausdorff approximation) and let $U(\hat E,\rho)$ denote the open $\rho$-neighborhood of $\hat E$.
A regular fiber is also defined for $E$ as a fiber over a regular point of $E$.

\begin{thm}\label{thm:ext}
Let $E$ be a primitive extremal subset of $X$ containing no proper extremal subsets.
Let $F$ denote a regular fiber over $E$ in $M$.
Then there exists $\rho>0$ (independent of $M$) such that if $\mu$ is small enough, the following hold.
\begin{enumerate}
\item $\pi_i(U(\hat E,\rho),F)$ is isomorphic to $\pi_i(E)$ for all $i$.
\item there exists a spectral sequence $\{E_r\}_{r=1}^\infty$ converging to $H^\ast(U(\hat E,\rho))$ whose $E_2$ term is the \v Cech cohomology of a good cover of $E$ with values in a locally constant presheaf $\mathcal H^\ast$ with group $H^\ast(F)$.
\end{enumerate}
\end{thm}

\begin{rem}\label{rem:yam}
As a generalization of Yamaguchi's fibration theorem, there is actually a locally trivial fibration over a ``regular part'' of $E$ if the singularities of $M$ are not so bad (see \cite[4.6]{F:reg}).
\end{rem}

\begin{rem}\label{rem:pet2}
A related (but unpublished) result of Petrunin can be found in \cite[4.6]{A}.
\end{rem}

\begin{org}
The organization of this paper is as follows.
In Section \ref{sec:note}, we introduce some notation and conventions.
In Section \ref{sec:pre}, we recall some notions and results on Alexandrov spaces which will be used later.
Section \ref{sec:prf} contains the proof of the main theorem \ref{thm:main}.
We first prove Theorem \ref{thm:main}(2) by considering a lift of a good covering of the limit space $X$ to the collapsing space $M$.
Next we show the homotopy lifting property with respect to the good covering and prove Theorem \ref{thm:main}(1).
Section \ref{sec:gen} includes generalizations of these results to each primitive extremal subset $E$ of $X$.
In Section \ref{sec:pte}, we extend the notion of a good point of an Alexandrov space defined by Perelman \cite{Per:col} to each primitive extremal subset and generalize the main result of \cite{Per:col}.
In Section \ref{sec:cove}, we study the relation between a lift of a good covering of $X$ and extremal subsets of $X$ to prove Theorem \ref{thm:ext}.
\end{org}

\begin{ack}
The author would like to thank Prof.\ Takao Yamaguchi for his useful advice and constant encouragement.
He is also grateful to the referee for carefully reading the manuscript.
\end{ack}

\section{Notation and conventions}\label{sec:note}

Throughout this paper, an upper bound $n$ for dimension and a lower bound $\kappa$ for curvature are fixed and omitted unless otherwise stated.

As in Section \ref{sec:intro}, $M$ and $X$ often denote a collapsing space and a limit space, respectively.
More precisely, $X$ denotes a fixed $k$-dimensional compact Alexandrov space, where $k<n$, and $M$ denotes a variable $n$-dimensional Alexandrov space that is $\mu$-close to $X$ for sufficiently small $\mu>0$ depending only on $X$.
We always fix $\mu$-approximations $\theta:M\to X$ and $\eta:X\to M$ such that $|\theta\circ\eta,\id_X|<\mu$.
Moreover, we use the hat symbol $\hat\ $ to indicate natural lifts of objects on $X$ to $M$ with respect to these approximations.
For example $\hat p\in M$ denotes $\eta(p)$ for $p\in X$ and $\hat A\subset M$ denotes $\overline{\eta(A)}$ for $A\subset X$ closed.
Similarly, if $f:X\to\mathbb R$ is constructed out of distance functions, then $\hat f:M\to\mathbb R$ denotes a function constructed by the same formula for the lifting points or sets.
Furthermore, if $K\subset X$ is defined by some inequalities of such functions, then $\hat K\subset M$ denotes a subset defined by the same inequalities for the lifting functions.

$\pi_i$ denotes the $i$-th homotopy group and $H^\ast$ denotes the singular cohomology group (with arbitrary coefficients).
$I^k(v,r)$ denotes the closed $r$-neighborhood of $v\in\mathbb R^k$ with respect to the maximum norm.
Unless otherwise stated, $I$ denotes the unit interval and $I^i$ denotes the $i$-dimensional unit cube.

In a metric space, $B(p,r)$ denotes the open $r$-ball around a point $p$ and $U(A,r)$ denotes the open $r$-neighborhood of a subset $A$.
The notation $A\Subset B$ for subsets $A$ and $B$ means that the closure of $A$ is contained in the interior of $B$.
The symbols $\mathring\ $ and $\bar\ $ (or $\Cl$) indicate the interior and closure, respectively.

\section{Preliminaries}\label{sec:pre}

In this section, we recall some notions and results on Alexandrov spaces which will be used later.
Section \ref{sec:alex} includes some standard facts and notation.
Section \ref{sec:adm} concerns regular admissible maps defined by Perelman \cite{Per:mor} and contains the fibration theorem for them.
We also recall the notion of a canonical neighborhood.
Section \ref{sec:ext} discusses extremal subsets defined by Perelman-Petrunin \cite{PP:ext} and the stratification of Alexandrov spaces.
Section \ref{sec:good} deals with the good coverings of Alexandrov spaces introduced by Mitsuishi-Yamaguchi \cite{MY:good}.
Semiconcave functions and their gradient flows are also discussed.
Section \ref{sec:ser} is a brief summary of the results of Perelman \cite{Per:col} mentioned in Section \ref{sec:intro}.

\subsection{Alexandrov spaces}\label{sec:alex}

Here we recall some basic facts on Alexandrov spaces.
See \cite{BGP} or \cite{BBI} for more details.

Let $M$ be an $n$-dimensional Alexandrov space with curvature $\ge\kappa$.
For a geodesic triangle $\triangle pqr$ in $M$, we denote by $\tilde\triangle pqr$ a comparison triangle with the same sidelengths on the $\kappa$-plane.
Then, by definition, the natural correspondence $\triangle pqr\to\tilde\triangle pqr$ is nonexpanding.
Let $\angle qpr$ denote the angle of $\triangle pqr$ at $p$ and $\tilde\angle qpr$ the corresponding angle of $\tilde\triangle pqr$.
Then we have $\angle qpr\ge\tilde\angle qpr$.

Let $\Sigma_p$ denote the space of directions at $p$ and $T_p$ the tangent cone at $p$.
Then $\Sigma_p$ is an $(n-1)$-dimensional Alexandrov space of curvature $\ge1$ and $T_p$ is an $n$-dimensional Alexandrov space of curvature $\ge 0$.
For $p,q\in M$, we denote by $\uparrow_p^q\in\Sigma_p$ one of the directions of shortest paths from $p$ to $q$, and by $\Uparrow_p^q\subset\Sigma_p$ the set of all directions of shortest paths from $p$ to $q$.
For $u,v\in T_p$, we define their scalar product $\langle u,v\rangle:=|u||v|\cos\angle(u,v)$, where the absolute value $|\cdot|$ denotes the distance from the vertex of $T_p$.
Let $\dist_q$ denote the distance function $|q\cdot|$ from $q$.
Then the first variation formula states
\[d_p\dist_q=-\langle\Uparrow_p^q,\cdot\rangle\]
on $T_p$.
Similarly, for a closed subset $A\subset M$, we use the notation $\uparrow_p^A$, $\Uparrow_p^A$, and $\dist_A$.
Then we have $d_p\dist_A=-\langle\Uparrow_p^A,\cdot\rangle$.

The boundary $\partial M$ of $M$ is defined inductively in such a way that $p\in\partial M\Leftrightarrow\partial\Sigma_p\neq\emptyset$.
The double $\tilde M$ of $M$, that is, the space consisting of two copies of $M$ glued together along their common boundaries, is again an Alexandrov space with the same lower curvature bound and without boundary (\cite[\S5]{Per:alex}).

The class of $n$-dimensional Alexandrov spaces with curvature $\ge\kappa$ and diameter $\le D$ is precompact with respect to the Gromov-Hausdorff topology.
Furthermore, the limit of a sequence of such Alexandrov spaces is an Alexandrov space with curvature $\ge\kappa$ and dimension $\le n$.

\subsection{Regular admissible maps}\label{sec:adm}

Here we define regular admissible maps and recall the fibration theorem and the canonical neighborhood theorem for them, proved by Perelman \cite{Per:mor}.
We also refer to \cite{Per:col} and \cite{K:stab}.

A typical example of a regular admissible map is a distance map $f=(|a_1\cdot|,\dots,|a_k\cdot|)$ around $p\in M$ such that $\tilde\angle a_ipa_j>\pi/2$ and $\tilde\angle a_ipw>\pi/2$ for all $i\neq j$ and some $w\in M$.
The actual definition is more general in order to prove the fibration theorem.
Here we present the definition in \cite{Per:col} (cf.\ \cite{K:stab}) rather than the original one in \cite{Per:mor} (the former is liftable with respect to the Gromov-Hausdorff convergence).
As before, $M$ denotes an $n$-dimensional Alexandrov space.

\begin{dfn}\label{dfn:adm}
A map $\bar f:M\to\mathbb R^k$ is said to be \textit{admissible} in a domain $U\subset M$ if it can be represented as $\bar f=H\circ f$ in $U$, where $H$ is a bi-Lipschitz homeomorphism between open subsets of $\mathbb R^k$, and the coordinate functions $f_i$ of $f$ have the form
\[f_i=\sum_{\alpha=1}^{N_i}\varphi_{i\alpha}(|q_{i\alpha}\cdot|),\]
where $q_{i\alpha}\in M$ and $\varphi_{i\alpha}$ are smooth, increasing, concave functions.

An admissible map $\bar f=H\circ f$ on $U$ is said to be \textit{regular} at $p\in U$ if it satisfies that
\begin{enumerate}
\item $\sum_{\substack{1\le\alpha\le N_i\\ 1\le\beta\le N_j}}\varphi_{i\alpha}'(|q_{i\alpha}p|)\varphi_{j\beta}'(|q_{j\beta}p|)\cos\tilde\angle q_{i\alpha}pq_{j\beta}<0$
 for all $i\neq j$;
\item there exists $\xi\in\Sigma_p$ such that $f_i'(\xi)>0$ for all $i$.
\end{enumerate}
\end{dfn}

If an admissible map $\bar f:U\to\mathbb R^k$ is regular at $p\in U$, then $k\le n$ and it is $c$-open near $p$ for some $c>0$ in the sense that $\bar f(B(q,r))\supset B(\bar f(q),cr)$ for any $q\in U$ near $p$ and sufficiently small $r>0$.
In particular, if $k=n$, then $\bar f$ is a local bi-Lipschitz homeomorphism near $p$.
Furthermore, the following fibration theorem holds.

\begin{thm}[\cite{Per:mor}]\label{thm:adm}
A proper regular admissible map on a domain $U\subset M$ is a locally trivial fibration.
\end{thm}

The proof of the above theorem is carried out by reverse induction on $k$, using the results of Siebenmann \cite{Si}.
The key ingredient in the proof of the induction step is the next canonical neighborhood theorem.

Let $\bar f:M\to\mathbb R^k$ be an admissible map regular at $p\in M$.
We say that $\bar f$ is \textit{complementable} at $p$ if there exists a function $\bar f_{k+1}:M\to\mathbb R$ such that $(\bar f,\bar f_{k+1}):M\to\mathbb R^k$ is admissible and regular near $p$.
Otherwise the following holds.

\begin{thm}[\cite{Per:mor}]\label{thm:nbhd}
Let $\bar f:M\to\mathbb R^k$ be an admissible map regular and incomplementable at $p\in M$.
Then there exist a neighborhood $U$ of $p$ and a function $\bar f_{k+1}:U\to\mathbb R$ such that $(\bar f,\bar f_{k+1}):U\to\mathbb R^{k+1}$ is admissible and the following are satisfied.
\begin{enumerate}
\item $\bar f_{k+1}\le 0$ on $U$ and $\bar f_{k+1}(p)=0$;
\item for sufficiently small $a>0$, the set
\[K(p,a):=\left\{x\in U\mid |\bar f(x)-\bar f(p)|\le a,\ \bar f_{k+1}(x)\ge-a\right\}\]
is compact, where we use the maximum norm of $\mathbb R^k$;
\item $\bar f$ is regular on $K(p,a)$, and $(\bar f,\bar f_{k+1})$ is regular on $K(p,a)\setminus\bar f_{k+1}^{-1}(0)$;
\item $\bar f$ restricted to $K(p,a)\cap\bar f_{k+1}^{-1}(0)$ is a bijection onto $I^k(\bar f(p),a)$.
\end{enumerate}
Note that if $k=n$, then one can put $\bar f_{n+1}:\equiv 0$.
\end{thm}

\begin{dfn}\label{dfn:nbhd}
The set $K(p,a)$ satisfying the conclusions (1)--(4) of Theorem \ref{thm:nbhd} is called a (closed) \textit{canonical neighborhood} of $p$ with respect to the admissible map $(\bar f,\bar f_{k+1})$ (note that we do not require that $\bar f$ is incomplementable at $p$).
Furthermore, we say that $K(p,a)$ \textit{respects} an admissible map $\bar g:M\to\mathbb R^l$ if the first $l$-coordinates of $\bar f$ coincide with $\bar g$.
\end{dfn}

In particular, given a regular admissible map, there exists a canonical neighborhood of the regular point respecting the given map.

\begin{rem}\label{rem:nbhd}
The fibration theorem \ref{thm:adm} yields a homeomorphism
\[\Phi:K(p,a)\to I^k(\bar f(p),a)\times(\bar f^{-1}(\bar f(p))\cap K(p,a))\]
respecting $\bar f$ and $\bar f_{k+1}$, that is, $\bar f=\pr_1\circ\Phi$ and $\bar f_{k+1}=\bar f_{k+1}\circ\pr_2\circ\Phi$, where $\pr_i$ denotes the projection onto the $i$-th factor.
Furthermore the fiber $\bar f^{-1}(\bar f(p))\cap K(p,a)$ has a conical structure with vertex $p$.
\end{rem}

\subsection{Extremal subsets}\label{sec:ext}

Here we discuss extremal subsets defined by Perelman-Petrunin \cite{PP:ext} and the stratification of Alexandrov spaces.
We also refer to \cite[\S4]{Pet:semi}, \cite[\S9]{K:stab}, and \cite{F:reg}.

Let $M$ be an Alexandrov space and consider a distance function $\dist_q=|q\cdot|$ from $q\in M$.
Recall that the first variation formula states $d_p\dist_q=-\langle\Uparrow_p^q,\cdot\rangle$ on $T_p$ for any $p\in M\setminus\{q\}$.
A \textit{critical point} of $\dist_q$ is a point $p$ such that $d_p\dist_q$ is nonpositive on $T_p$.

\begin{dfn}\label{dfn:ext}
A closed subset $E$ of $M$ is said to be \textit{extremal} if it satisfies the following condition (E).
\begin{itemize}
\item[(E)] If $p\in E$ is a local minimum point of $\dist_q|_E$, where $q\notin E$, then it is a critical point of $\dist_q$.
\end{itemize}
Note that the empty set $\emptyset$ and $M$ itself are regarded as extremal subsets of $M$.
\end{dfn}

\begin{rem}\label{rem:ext}
Extremal subsets have an alternative definition in terms of gradient flows of semiconcave functions: a subset of an Alexandrov space is extremal if and only if it is invariant under any gradient flow.
See the next section for the definition of gradient flows.
\end{rem}

\begin{ex}\label{ex:ext}
\begin{enumerate}
\item A one-point subset $\{p\}$ of $M$ is extremal if and only if the diameter of $\Sigma_p$ is not greater than $\pi/2$.
\item Fix a topological cone $K$ with vertex $o$.
Then the closure of the set of points in $M$ whose tangent cones are pointed homeomorphic to $(K,o)$ is an extremal subset.
In particular the boundary of $M$ is an extremal subset.
This follows from the fibration theorem \ref{thm:adm} and the stability theorem (\cite{Per:alex}, cf.\ \cite{K:stab}).
See Lemma \ref{lem:s} and Example \ref{ex:s} for the generalization.
\item Let a compact group $G$ act on $M$ by isometries.
Then the quotient $M/G$ is again an Alexandrov space with the same lower curvature bound.
Let $H$ be a closed subgroup of $G$.
Then the natural projection of the fixed point set of $H$ is an extremal subset of $M/G$.
\end{enumerate}
\end{ex}

The Hausdorff dimension coincides with the topological dimension for extremal subsets.
However, the local dimension may not be constant in general (this is not the case for primitive extremal subsets defined below).
Tangent cones and spaces of directions are also defined for extremal subsets.

The results of the previous section can be generalized to extremal subsets.
For example, a regular admissible map to $\mathbb R^k$ restricted to an extremal subset $E$ is open, and moreover it is a local homeomorphism if $k$ is the local dimension of $E$.
Furthermore, the fibration theorem \ref{thm:adm} and the canonical neighborhood theorem \ref{thm:nbhd} also hold for extremal subsets.
These results will be stated more precisely when necessary, especially in Section \ref{sec:pte}.

The union, the intersection, and the closure of the difference of two extremal subsets are also extremal.
Moreover, the collection of extremal subsets in an Alexandrov space is locally finite in a certain sense.
These two facts lead to the following definition.

\begin{dfn}\label{dfn:prim}
An extremal subset is said to be \textit{primitive} if it contains no proper extremal subsets with nonempty relative interior.
The \textit{main part} $\mathring E$ of a primitive extremal subset $E$ is the relative complement of all proper extremal subsets in $E$.
\end{dfn}

Any extremal subset can be uniquely represented as a union of primitive extremal subsets with nonempty relative interior, and the main part of a primitive extremal subset is open and dense in it.
Clearly the main parts of all primitive extremal subsets in an Alexandrov space form a disjoint covering of that space.
Furthermore, the fibration theorem for extremal subsets mentioned above shows the following.

\begin{thm}[\cite{PP:ext}]\label{thm:prim}
The main parts of primitive extremal subsets are topological manifolds.
In particular, they define a stratification of an Alexandrov space into topological manifolds.
\end{thm}

\subsection{Good coverings}\label{sec:good}

Here we recall the notion of a good covering of an Alexandrov space introduced by Mitsuishi-Yamaguchi \cite{MY:good}.
We also refer to \cite{Pet:semi} for semiconcave functions and their gradient flows.

Let us first define semiconcave functions.
Note that the definition of Alexandrov spaces says that their distance functions are more concave than those of the space of constant curvature.

Let $M$ be an Alexandrov space.
First suppose $M$ has no boundary.
A (locally Lipschitz) function $f:U\to\mathbb R$ on an open subset $U\subset M$ is said to be \textit{$\lambda$-concave} if for any shortest path $\gamma(t)$ in $U$ parametrized by arclength, the function $f\circ\gamma(t)-(\lambda/2)t^2$ is concave in the usual sense.
In case $M$ has boundary, we say that $f$ is \textit{$\lambda$-concave} if its natural extension to the double of $M$ is $\lambda$-concave in the above sense.
A function $f:U\to\mathbb R$ is said to be \textit{semiconcave} if for any $x\in U$ there exists $\lambda_x\in\mathbb R$ such that $f$ is $\lambda_x$-concave on a neighborhood of $x$.
We also say that $f$ is \textit{strictly concave} if it is $\lambda$-concave for some $\lambda<0$.

The key step in the proof of the canonical neighborhood theorem \ref{thm:nbhd} is to construct a strictly concave admissible function out of distance functions from points in the regular direction of the given admissible map.
By taking a minimum of such functions constructed in sufficiently many directions around a point, one can obtain the following (see \cite[\S4]{K:reg} or the proof of Lemma \ref{lem:sc} for details).

\begin{thm}\label{thm:sc}
For any $p\in M$, there exists a strictly concave function $h$ defined on a neighborhood of $p$ and attaining its unique maximum at $p$.
\end{thm}

Using such functions, Mitsuishi-Yamaguchi \cite{MY:good} defined a good covering of an Alexandrov space and proved the following theorem.
Here we define it more directly for our application.

\begin{dfn}\label{dfn:good}
A locally finite open covering $\mathcal U=\{U_\alpha\}_\alpha$ of $M$ is said to be \textit{good} if each $U_\alpha$ is a strict superlevel set of a strictly concave function $h_\alpha$ on $B(p_\alpha,r_\alpha)$ constructed in Theorem \ref{thm:sc} and satisfies $U_\alpha\subset B(p_\alpha,r_\alpha/2)$.
\end{dfn}

\begin{thm}[\cite{MY:good}]\label{thm:good}
Let $\mathcal U=\{U_\alpha\}_\alpha$ be a good covering of $M$.
Then every nonempty intersection of $U_\alpha$'s is a convex, conical, strongly Lipschitz contractible bounded domain.
\end{thm}

Here an open subset $U$ is called \textit{conical} and \textit{strongly Lipschitz contractible} if there exists $p\in U$ such that $U$ is pointed homeomorphic to the tangent cone at $p$ and admits a deformation retraction to $p$ that is Lipschitz and monotonically decreases the distance to $p$ (see the original paper for the precise definitions).

\begin{rem}\label{rem:good}
\begin{enumerate}
\item The condition $U_\alpha\subset B(p_\alpha,r_\alpha/2)$ in the above definition is only needed for the convexity.
\item The original definition of a good covering in \cite{MY:good} is a locally finite open covering satisfying the conclusion of the above theorem.
\end{enumerate}
\end{rem}

For the proof, we need the notion of the gradient of semiconcave functions.
Let $f:U\to\mathbb R$ be a semiconcave function.
Note that the directional derivative $d_pf:T_p\to\mathbb R$ exists at any $p\in U$.
The \textit{gradient} $\nabla_pf\in T_p$ of $f$ at $p$ is characterized by the properties
\begin{enumerate}
\item $d_pf(v)\le\langle\nabla_pf,v\rangle$ for any $v\in T_p$;
\item $d_pf(\nabla_pf)=|\nabla_pf|^2$
\end{enumerate}
(see Section \ref{sec:alex} for the definitions of the scalar product and the absolute value).
More specifically, if $\max d_pf|_{\Sigma_p}>0$ then $\nabla_pf=d_pf(\xi)\xi$, where $\xi\in\Sigma_p$ is the unique maximum point of $d_pf|_{\Sigma_p}$; otherwise $\nabla_pf=o$.

The conical structure in Theorem \ref{thm:good} is provided by the fibration theorem.
A semiconcave function $f:U\to\mathbb R$ is said to be \textit{regular} at $p\in U$ if $|\nabla_p f|>0$ (cf.\ Definition \ref{dfn:adm}).
The fibration theorem \ref{thm:adm} is generalized to regular semiconcave functions (see \cite[\S2]{Per:dc}, \cite[\S8]{Pet:semi}, \cite[1.5]{MY:good}).
Since a strictly concave function is regular except at the maximum point, this theorem yields the conical structure of a superlevel set.

On the other hand, the strong Lipschitz contraction is given by the gradient flows of semiconcave functions.
A curve $\alpha(t)$ is called a \textit{gradient curve} of a semiconcave function $f$ if its right tangent vector $\alpha^+(t)$ uniquely exists and equals $\nabla_{\alpha(t)}f$ for any $t$.
The gradient curve starting at a point is unique if it exists.
Moreover, for any standard semiconcave function such as a distance function, it actually exists for all $t\ge0$. 
The \textit{gradient flow} $\Phi_t$ ($t\ge0$) of $f$ is defined by $\Phi_t(p)=\alpha_p(t)$, where $\alpha_p$ is the gradient curve of $f$ starting at $p$.
The strong Lipschitz contraction of Theorem \ref{thm:good} is constructed by gluing some gradient flows.

As mentioned in Remark \ref{rem:ext}, extremal subsets are invariant under any gradient flow.
Therefore the strong Lipschitz contractions of a good covering preserve all extremal subsets.

It is also worth mentioning that the strict concavity of the function $h$ of Theorem \ref{thm:sc} is liftable with respect to the noncollapsing convergence of Alexandrov spaces.
Therefore the goodness of a covering is liftable in this case, which yields the finiteness of Lipschitz homotopy types of noncollapsing Alexandrov spaces (\cite{MY:lip}).
In this paper, we study the relation between good coverings and collapsing.

\subsection{Serre fibration theorem}\label{sec:ser}

Here we summarize the results of Perelman \cite{Per:col} mentioned in Section \ref{sec:intro}.

We use the same notation as in Section \ref{sec:intro}: $X$ denotes a fixed $k$-dimensional compact Alexandrov space, where $k<n$, and $M$ denotes a variable $n$-dimensional Alexandrov space that is $\mu$-close to $X$.
We fix $\mu$-approximations $\theta:M\to X$ and $\eta:X\to M$.
Furthermore $\hat p\in M$ denotes a lift of $p\in X$.
Similarly, if $f:X\to\mathbb R^l$ is constructed out of distance functions like admissible maps, then $\hat f:M\to\mathbb R^l$ denotes its natural lift (see Section \ref{sec:note}).

Let us first define regular fibers.
Let $f:X\to\mathbb R^k$ be an admissible map regular at $p\in X$, where $k=\dim X$.
For some $a>0$, $f$ is regular on $f^{-1}(I^k(f(p),a))$ (in a small neighborhood of $p$) and the restriction $f:f^{-1}(I^k(f(p),a))\to I^k(f(p),a)$ is a homeomorphism.
We call such a pair $(f,p)$ \textit{fiber data} on $X$ and the set $f^{-1}(I^k(f(p),a))$ its \textit{coordinate neighborhood}.
Now suppose $\mu\ll a$ is small enough and lift the situation to $M$.
Since the lift $\hat f$ is also regular on $\hat f^{-1}(I^k(f(p),a))$, the fibration theorem \ref{thm:adm} implies that $\hat f:\hat f^{-1}(I^k(f(p),a))\to I^k(f(p),a)$ is a trivial bundle.
We call the fiber $\hat f^{-1}(f(p))$ a \textit{regular fiber} in $M$.
Note that fiber data is entirely defined on fixed $X$, whereas its regular fiber is defined in variable $M$.

Note that the set of points of fiber data is open, dense, and convex in $X$.
The convexity follows from Petrunin's result \cite{Pet:para} on parallel transport, which shows that the space of directions is invariant on the interior of a shortest path.
Indeed, a point $p$ is a part of fiber data if and only if $\Sigma_p$ contains $k+1$ directions making obtuse angles with each other (this follows from \cite[2.2]{Per:mor}).
Hence the above result together with the lower semicontinuity of the space of directions yields the convexity.

Furthermore, all regular fibers are homotopy equivalent in the following sense: given two fiber data on $X$, if $\mu$ is small enough, then the corresponding regular fibers in $M$ are homotopy equivalent.
We will later review the proof of this fact for our application (see Lemma \ref{lem:rf}).

Using regular fibers, Perelman defined the notion of a good point.

\begin{dfn}\label{dfn:pn}
A point $p\in X$ is said to be \textit{good} if it satisfies the following condition (PN).
\begin{itemize}
\item[(PN)]
For any $R>0$, there exists $\rho=\rho(p,R)$ such that for any fiber data $(f,q)$ with $q\in B(p,\rho)$, if $\mu$ is small enough, then $M$ contains a \textit{product neighborhood}, that is, a domain $U$ with $B(\hat p,\rho)\subset U\subset B(\hat p,R)$ such that the inclusion $\hat f^{-1}(f(q))\hookrightarrow U$ is a homotopy equivalence.
\end{itemize}
\end{dfn}

\begin{rem}\label{rem:pn}
To be precise, the original definition in \cite{Per:col} only requires that the inclusion $\hat f^{-1}(f(q))\hookrightarrow U$ induces isomorphisms of homotopy and homology groups.
However, one can easily check that it can be replaced with a homotopy equivalence (actually a deformation retract as we will see in Section \ref{sec:pte}).
\end{rem}

Note that if $p\in X$ is a regular point of an admissible map $f:X\to\mathbb R^k$, where $k=\dim X$, then it is a good point.
Indeed one can take $\hat f^{-1}(\mathring I^k(f(p),a))$ as a product neighborhood, where $\rho\ll a\ll R$.
In particular the set of good points is dense in $X$.

Perelman proved the following theorem.

\begin{thm}[\cite{Per:col}]\label{thm:ser}
The closure of the set of bad points in $X$, if nonempty, is a proper extremal subset.
In particular, if $X$ has no proper extremal subsets, then every point is good.
\end{thm}

In case $X$ has no proper extremal subsets, this implies the following homotopy lifting property with respect to the approximation $\theta:M\to X$.
All maps below (but $\theta$) are assumed to be continuous.

\begin{cor}\label{cor:ser}
Suppose $X$ has no proper extremal subsets.
Fix an integer $j\ge1$. 
Then, for any $R>0$, there exists $r>0$ such that the following holds provided $\mu$ is small enough:
Let $K$ be a finite simplicial complex of dimension $\le j$.
Suppose $\sigma:K\times I\to X$ and $\hat\sigma:K\times\{0\}\to M$ satisfy $|\sigma,\theta\hat\sigma|<r$ on $K\times\{0\}$.
Then, $\hat\sigma$ can be extended to a map on $K\times I$ so that $|\sigma,\theta\hat\sigma|<R$ on $K\times I$.
\end{cor}

Theorem \ref{thm:per} follows from this corollary.
We shall remove the restriction $j$ on the dimension in the next section by using good coverings.

\section{Proof of the main theorem}\label{sec:prf}

In this section we prove Theorem \ref{thm:main}.
We first prove Theorem \ref{thm:main}(2) by considering a lift of a good covering of the limit space $X$ to the collapsing space $M$.
We show that all elements of the nerve of the lifted cover have the homotopy type of a regular fiber (Proposition \ref{prop:shf}).
Next, using it, we show the homotopy lifting property with respect to the good covering (Proposition \ref{prop:lift}).
In particular, it is independent of the dimension of simplicial complexes and enables us to prove Theorem \ref{thm:main}(1).

Throughout this section, $X$ denotes a fixed $k$-dimensional compact Alexandrov space, where $k<n$, and $M$ denotes a variable $n$-dimensional Alexandrov space that is $\mu$-close to $X$.
We fix $\mu$-approximations $\theta:M\to X$ and $\eta:X\to M$ such that $|\theta\circ\eta,\id_X|<\mu$.
The hat symbol $\hat\ $ is used to indicate natural lifts of objects on $X$ to $M$ (see Section \ref{sec:note}).
Furthermore, we often assume that $X$ has no proper extremal subsets.
Note that this assumption implies that every point of $X$ is good in the sense of Definition \ref{dfn:pn} (see Theorem \ref{thm:ser}).

Let us first recall the proof of the following fact mentioned in Section \ref{sec:ser}.
Here the absence of proper extremal subsets is not needed.

\begin{lem}[\cite{Per:col}]\label{lem:rf}
Let $(f_1,p_1)$ and $(f_2,p_2)$ be fiber data on $X$.
Then, if $\mu$ is small enough, the corresponding regular fibers $F_1$ and $F_2$ in $M$ are homotopy equivalent.
\end{lem}

\begin{proof}
First consider the special case $p_1=p_2=p$.
Recall that the fiber data $(f_i,p)$ ($i=1,2$) define their coordinate neighborhoods $f_i^{-1}(I^k(f_i(p),a_i))$ for some $a_i>0$ and that their lifts $\hat f_i:\hat f_i^{-1}(I^k(f_i(p),a_i))\to I^k(f_i(p),a_i)$ are trivial bundles if $\mu\ll a_i$ is small enough (see Section \ref{sec:ser}).
Fix trivializations $\hat f_i^{-1}(I^k(f_i(p),a_i))\cong F_i\times I^k(f_i(p),a_i)$ with respect to $\hat f_i$ and define (deformation) retractions $R_i:\hat f_i^{-1}(I^k(f_i(p),a))\to F_i$ by using the radial (deformation) retractions of $I^k(f_i(p),a)$ to $\{f_i(p)\}$.
Then, it is easy to see that the restrictions $R_1|_{F_2}:F_2\to F_1$ and $R_2|_{F_1}:F_1\to F_2$ are homotopy equivalences.

Next consider the general case.
Connect $p_1$ and $p_2$ by a shortest path $\gamma$.
Then, by the convexity mentioned in Section \ref{sec:ser}, every point on $\gamma$ is also a regular point of some admissible map to $\mathbb R^k$.
Take points $p_1=q_1$, $q_2$, \dots, $q_N=p_2$ on $\gamma$ in this order and admissible maps $g_i:X\to\mathbb R^k$ regular at $q_i$ so that their coordinate neighborhoods $g_i^{-1}(I^k(g_i(q_i),a_i))$ cover $\gamma$ (we may assume $g_1=f_1$, $g_N=f_2$).
Pick points $r_i$ from the intersection of the coordinate neighborhoods of $q_i$ and $q_{i+1}$ (put $r_0:=q_1$, $r_N:=q_N$).
Now suppose $\mu\ll a_i$ is small enough and fix trivializations $\hat g_i^{-1}(I^k(g_i(q_i),a_i))\cong\hat g_i^{-1}(g_i(q_i))\times I^k(g_i(q_i),a_i)$ with respect to $\hat g_i$.
Set $G_i^j:=\hat g_i^{-1}(g_i(r_j))$ for $j=i-1,i$.
Then $G_i^{i-1}$ is homeomorphic to $G_i^i$ via the above trivialization.
Furthermore, the argument in the special case shows that $G_i^i$ is homotopy equivalent to $G_{i+1}^i$.
This completes the proof.
\end{proof}

Let $\mathcal U=\{U_\alpha\}_{\alpha=1}^N$ be a finite good covering of $X$.
We may assume that $U_\alpha=\{h_\alpha>0\}$, where $h_\alpha$ is a strictly concave function constructed by Theorem \ref{thm:sc} and attaining its unique maximum at $p_\alpha\in U_\alpha$.
Furthermore, for a subset $A\subset\{1,\dots,N\}$, we define
\[h_A:=\min_{\alpha\in A}h_\alpha,\quad U_A:=\bigcap_{\alpha\in A}U_\alpha=\{h_A>0\}\]
and denote by $p_A\in U_A$ the unique maximum point of $h_A$ (if $U_A$ is nonempty).
Note that the strict concavity of $h_A$ implies that it is regular on $U_A\setminus\{p_A\}$, i.e.\ $|\nabla h_A|>0$.
Moreover, if $h_A$ is $(-c_A)$-concave on $U_A$, where $c_A>0$, then it is $(c_Ar/2)$-regular on $U_A\setminus B(p_A,r)$, i.e.\ $|\nabla h_A|>c_Ar/2$, since
\[h_A(\uparrow_x^{p_A})\ge\frac{h_A(p_A)-h_A(x)}{|p_Ax|}+\frac{c_A}2{|p_Ax|}>\frac{c_A}2r\]
for any $x\in U_A\setminus B(p_A,r)$.

Now, suppose $\mu$ is small enough and lift the situation to $M$.
Let $\hat h_\alpha$ be a natural lift of $h_\alpha$ and set $\hat U_\alpha:=\{\hat h_\alpha>0\}$.
Then  $\hat{\mathcal U}=\{\hat U_\alpha\}_{\alpha=1}^N$ is also an open covering of $M$.
Define $\hat h_A$ and $\hat U_A$ in the same way  as above.
Note that $\hat h_A$ is also $(c_Ar/2)$-regular on $\hat U_A\setminus B(\hat p_A,r)$ for any fixed $r>0$ provided $\mu\ll r$, by the lower semicontinuity of the absolute value of the gradient.
In particular, we may assume that $\mathcal U$ and $\hat{\mathcal U}$ have the same nerve, namely $U_A\neq\emptyset\Leftrightarrow\hat U_A\neq\emptyset$ for any $A$.
Indeed, an argument using the gradient flows shows that a slightly smaller cover $\{\{\hat h_\alpha>\varepsilon\}\}_{\alpha=1}^N$ for some fixed $\varepsilon>0$ has the same nerve as $\hat{\mathcal U}$.

Theorem \ref{thm:main}(2) immediately follows from the next proposition.

\begin{prop}\label{prop:shf}
Let $\mathcal U=\{U_\alpha\}_{\alpha=1}^N$ be a good covering of $X$ as above.
Suppose $X$ has no proper extremal subsets and $\mu$ is small enough.
Then, for any $A$, the lift $\hat U_A$ has the homotopy type of a regular fiber.
Furthermore, if $A\subset A'$, then the inclusion $\hat U_{A'}\hookrightarrow\hat U_A$ is a homotopy equivalence.
\end{prop}

\begin{proof}
The notation $U\Subset V$ means that the closure of $U$ is contained in the interior of $V$.
We first prove that $\hat U_A$ has the homotopy type of a regular fiber.
Recall that our assumption implies that every point of $X$ is good in the sense of Definition \ref{dfn:pn} (see Theorem \ref{thm:ser}).
Let $R>0$ be so small that $B(p_A,R)\Subset U_A$ and take $\rho=\rho(p_A,R)$ by the condition (PN).
Choose $\varepsilon>0$ so that $U_A(\varepsilon):=\{h_A>h_A(p_A)-\varepsilon\}\Subset B(p_A,\rho)$.
Take fiber data $(f_A,q_A)$ with $q_A\in U_A(\varepsilon)$.
In short, we have the inclusions
\[q_A\in U_A(\varepsilon)\Subset B(p_A,\rho)\Subset B(p_A,R)\Subset U_A.\]
Now consider their lifts in $M$.
Let $F_A=\hat f_A^{-1}(f_A(q_A))$ be a regular fiber.
Then the condition (PN) yields a product neighborhood $V_A$ for $F_A$ such that $B(\hat p_A,\rho)\subset V_A\subset B(\hat p_A,R)$.
Thus, if $\mu$ is small enough, we have the inclusions
\[F_A\subset\Cl(\hat U_A(\varepsilon))\subset V_A\subset\hat U_A,\]
where $\hat U_A(\varepsilon):=\{\hat h_A>h_A(p_A)-\varepsilon\}$ and $\Cl$ denotes the closure.
Here $F_A\hookrightarrow V_A$ is a homotopy equivalence.
Furthermore, the regularity of $\hat h_A$ implies that $\Cl(\hat U_A(\varepsilon))$ is a deformation retract of $\hat U_A$ by the gradient flow (or the fibration theorem for regular semiconcave functions, see \cite[\S2]{Per:dc}, \cite[\S8]{Pet:semi}, \cite[1.5]{MY:good}).
Therefore $F_A\hookrightarrow\hat U_A$ is a homotopy equivalence.

Next we prove that $\hat U_{A'}\hookrightarrow\hat U_A$ is a homotopy equivalence for any $A\subset A'$.
Take regular fibers $F_A$ and $F_{A'}$ as above such that the inclusions $F_A\hookrightarrow\hat U_A$ and $F_{A'}\hookrightarrow\hat U_{A'}$ are homotopy equivalences.
Let $\Phi:F_{A'}\to F_A$ be a homotopy equivalence constructed in the proof of Lemma \ref{lem:rf}.
Note that the construction is done on an arbitrarily small neighborhood of a shortest path connecting $q_{A'}$ and $q_A$, which can be contained in $U_A$ by the convexity.
Thus one can easily check that the composition $F_{A'}\xrightarrow\Phi F_A\hookrightarrow\hat U_A$ is homotopic to the inclusion $F_{A'}\hookrightarrow\hat U_A$ in $\hat U_A$.
Hence $F_{A'}\hookrightarrow\hat U_A$ is also a homotopy equivalence and so is $\hat U_{A'}\hookrightarrow\hat U_A$.
\end{proof}

\begin{proof}[Proof of Theorem \ref{thm:main}(2)]
Fix a regular fiber $F$.
We have now constructed an open cover $\hat{\mathcal U}$ of $M$ with the same nerve as a good cover $\mathcal U$ of $X$ which satisfies that
\begin{itemize}
\item the cohomology group $H^\ast(\hat U_A)$ is isomorphic to $H^\ast(F)$ for any $A$;
\item the restriction $H^\ast(\hat U_{A'})\to H^\ast(\hat U_A)$ is an isomorphism for any $A\subset A'$.
\end{itemize}
In other words, $H^\ast(\hat U_A)$ defines a locally constant presheaf on $\mathcal U$ with group $H^\ast(F)$.
Thus the standard argument of spectral sequences for fiber bundles can be applied to compute the cohomology of $M$ (see \cite[Chapter III]{BT} for instance).
This completes the proof.
\end{proof}

Next we prove Theorem \ref{thm:main}(1).
Using Proposition \ref{prop:shf}, we first show the homotopy lifting property with respect to a good covering.
All maps below but Gromov-Hausdorff approximations are assumed to be continuous.
Let $I$ denote the unit interval.

\begin{prop}[cf.\ {\cite[2.4]{Per:col}}]\label{prop:lift}
Let $\mathcal U=\{U_\alpha\}_{\alpha=1}^N$ be a good covering of $X$.
Suppose $X$ has no proper extremal subsets and $\mu$ is small enough.
Let $K$ be a finite simplicial complex.
Suppose $\sigma:K\times I\to X$ and $\hat\sigma:K\times\{0\}\to M$ satisfy that
\begin{itemize}
\item for any simplex $\Delta\subset K$, there exists $\alpha$ such that $\sigma(\Delta\times I)\subset U_\alpha$;
\item for any $\Delta\subset K$, if $\sigma(\Delta\times\{0\})\subset U_\alpha$, then $\hat\sigma(\Delta\times\{0\})\subset\hat U_\alpha$.
\end{itemize}
Then, $\hat\sigma$ can be extended to a map on $K\times I$ so that
\begin{itemize}
\item for any $\Delta\subset K$, if $\sigma(\Delta\times I)\subset U_\alpha$, then $\hat\sigma(\Delta\times I)\subset\hat U_\alpha$;
\item for any $\Delta\subset K$, if $\sigma(\Delta\times\{1\})\subset U_\alpha$, then $\hat\sigma(\Delta\times\{1\})\subset\hat U_\alpha$.
\end{itemize}
Furthermore, if $L$ is a subcomplex of $K$ and $\hat\sigma$ is already defined on $L\times I$ so that the above conclusions hold for any $\Delta\subset L$, then we can extend it.
\end{prop}

\begin{proof}
The proof is by induction on the dimension of simplices as in \cite[2.4]{Per:col}.
Fix a simplex $\Delta\subset K\setminus L$.
By the induction hypothesis, $\hat\sigma$ is already defined on $\Delta\times\{0\}\cup\partial\Delta\times I$.
Let us extend it over $\Delta\times I$.
Let $A\subset A'\subset\{1,\dots,N\}$ be the maximal subsets satisfying
\[\sigma(\Delta\times I)\subset U_A,\quad\sigma(\Delta\times\{1\})\subset U_{A'}.\]
Then the induction hypothesis implies
\[\hat\sigma(\Delta\times\{0\}\cup\partial\Delta\times I)\subset\hat U_A,\quad\hat\sigma(\partial\Delta\times\{1\})\subset\hat U_{A'}.\]
Since $\hat U_{A'}\hookrightarrow\hat U_A$ is a homotopy equivalence by Proposition \ref{prop:shf}, the relative homotopy group $\pi_i(\hat U_A, \hat U_{A'})$ vanishes, where $i=\dim\Delta$.
Hence, if we regard $\hat\sigma|_{\Delta\times\{0\}\cup\partial\Delta\times I}$ as an element of $\pi_i(\hat U_A, \hat U_{A'})$ in a natural way, it is null-homotopic in $(\hat U_A, \hat U_{A'})$.
Using this homotopy, we can extend $\hat\sigma$ over $\Delta\times I$ so that $\hat\sigma(\Delta\times I)\subset\hat U_A$ and $\hat \sigma(\Delta\times\{1\})\subset\hat U_{A'}$.
\end{proof}

The above proposition immediately implies the homotopy lifting property with respect to the approximation $\theta:M\to X$, which is a refinement of Corollary \ref{cor:ser} in that it does not depend on the dimension of simplicial complexes.

\begin{cor}\label{cor:lift}
Suppose $X$ has no proper extremal subsets.
Then, for any $R>0$, there exists $r>0$ such that the following holds provided $\mu$ is small enough:
Let $K$ be a finite simplicial complex and $L$ a subcomplex.
Suppose $\sigma:K\times I\to X$ and $\hat \sigma:K\times\{0\}\cup L\times I\to M$ satisfy $|\sigma,\theta\hat \sigma|<r$ on $K\times\{0\}\cup L\times I$.
Then, $\hat\sigma$ can be extended to a map on $K\times I$ so that $|\sigma,\theta\hat\sigma|<R$ on $K\times I$.
\end{cor}

\begin{proof}
Let $\mathcal U=\{U_\alpha\}_{\alpha=1}^N$ be a good covering of $X$ as before.
We may assume that $\diam U_\alpha<R/2$ for any $\alpha$.
We may further assume that $\hat U_\alpha$ is slightly larger than $U_\alpha$, that is, $\hat U_\alpha=\{\hat h_\alpha>-\varepsilon\}$ for some fixed $0<\varepsilon\ll R$ whereas $U_\alpha=\{h_\alpha>0\}$ (this does not change the nerve as we have seen before Proposition \ref{prop:shf}).
Choose subdivisions $K'$ of $K$ and $0=t_0<t_1<\dots<t_J=1$ of $[0,1]$ so that for any $\Delta'\subset K'$ and $1\le j\le J$, there exists $\alpha$ such that $\sigma(\Delta'\times[t_{j-1},t_j])\subset U_\alpha$.
Furthermore, if $r$ and  $\mu$ are small enough compared with $\varepsilon$ (but independent of $\sigma$), we have
\[\sigma(\Delta'\times\{0\})\subset U_\alpha\Longrightarrow\hat\sigma(\Delta'\times\{0\})\subset\hat U_\alpha\]
for any $\Delta'\subset K'$ and
\begin{gather*}
\sigma(\Delta'\times[t_{j-1},t_j])\subset U_\alpha\Longrightarrow\hat\sigma(\Delta'\times[t_{j-1},t_j])\subset\hat U_\alpha,\\
\sigma(\Delta'\times\{t_j\})\subset U_\alpha\Longrightarrow\hat\sigma(\Delta'\times\{t_j\})\subset\hat U_\alpha.
\end{gather*}
for any $\Delta'\subset L$ and $1\le j\le J$.
Then, by Proposition \ref{prop:lift}, we can extend $\hat\sigma$ over $K'\times[t_0,t_1]$ so that
\begin{gather*}
\sigma(\Delta'\times[t_0,t_1])\subset U_\alpha\Longrightarrow\hat\sigma(\Delta'\times[t_0,t_1])\subset\hat U_\alpha,\\
\sigma(\Delta'\times\{t_1\})\subset U_\alpha\Longrightarrow\hat\sigma(\Delta'\times\{t_1\})\subset\hat U_\alpha
\end{gather*}
for any $\Delta'\subset K'$.
The second condition enables us to repeat this procedure to obtain $\hat\sigma:K\times I\to M$ satisfying
\[\sigma(\Delta'\times [t_{j-1},t_j])\subset U_\alpha\Longrightarrow\hat\sigma(\Delta'\times [t_{j-1},t_j])\subset\hat U_\alpha\]
for any $\Delta'\subset K'$ and $1\le j\le J$.
Since $\diam U_\alpha<R/2$, this implies $|\sigma,\theta\hat\sigma|<R$.
\end{proof}

For the proof of Theorem \ref{thm:main}(1), we need the following basic lemma, which also uses a good covering (here the absence of proper extremal subsets is not needed).

\begin{lem}\label{lem:hom}
There exists a constant $R_0>0$ such that for any finite simplicial complex $K$ and its subcomplex $L$, any two maps $\sigma_i:(K,L)\to(X,p)$ ($i=0,1$) that are $R_0$-close to each other are homotopic relative to $L$.
\end{lem}

\begin{proof}
Let $R_0$ be half the Lebesgue number of a good covering $\mathcal U$ of $X$.
Then, we can choose a subdivision $K'$ of $K$ so that for any $\Delta'\subset K'$, there exists $\alpha$ such that $\sigma_0(\Delta')\cup\sigma_1(\Delta')\subset U_\alpha$.
Using the contractibility of $\mathcal U$, we can define a homotopy $\sigma_t$ ($0\le t\le1$) inductively on the skeleta so that
\[\sigma_0(\Delta')\cup\sigma_1(\Delta')\subset U_\alpha\Longrightarrow\sigma_t(\Delta')\subset U_\alpha\]
for any $\Delta'\subset K'$ and $\sigma_t|_L\equiv p$.
\end{proof}

We are now in a position to prove Theorem \ref{thm:main}(1).
The proof is almost the same as in \cite{Per:col}, except that we use good coverings.

\begin{proof}[Proof of Theorem \ref{thm:main}(1)]
Fix fiber data $(f,p)$ and let $F=\hat f^{-1}(f(p))$ be its regular fiber.
We may assume $\hat p\in F$ by the $c(f)$-openness of $\hat f$, where $c(f)$ is a constant depending only on $f$.
We may further assume $\theta(F)=p$ since the bi-Lipschitzness of $f$ shows $\diam F\le c(f)\mu$.
Recall that there exists a trivialization $\hat f^{-1}(I^k(f(p),a)\cong F\times I^k(f(p),a)$ with respect to $\hat f$ for some fixed $a>0$.
In particular, $\hat f^{-1}(I^k(f(p),a))$ admits a deformation retraction onto $F$.

Fix $r_0>0$ such that $r_0<R_0/2$ and $r_0\ll a$, where $R_0$ is the constant of Lemma \ref{lem:hom}.
Suppose $\mu\ll r_0$ is small enough.
Let
\[\sigma:(I^i,\partial I^i)\to (X,p),\quad\hat\sigma:(I^i,\partial I^i,J^{i-1})\to(M,F,\hat p),\]
where $I^i=[0,1]^i$ and $J^{i-1}=\partial I^i\setminus I^{i-1}\times\{1\}$.
We say that $\hat\sigma$ is a \textit{lift} of $\sigma$, or $\sigma$ is a \textit{projection} of $\hat\sigma$, if $|\sigma,\theta\hat\sigma|<r_0$ for the approximation $\theta:M\to X$.
Then, the desired isomorphism $\pi_i(X,p)\cong\pi_i(M,F,\hat p)$ is given by the correspondence $[\sigma]\leftrightarrow[\hat\sigma]$.
It suffices to show the following four properties.

\begin{proof}[Existence of $\sigma$]\renewcommand{\qedsymbol}{}
Let us show the existence of a projection $\sigma$ for given $\hat\sigma$.
Take a good covering $\mathcal U$ of $X$ such that $\diam U_\alpha< r_0$ for any $\alpha$.
Suppose $\mu$ is less than half the Lebesgue number of $\mathcal U$.
Then we can choose a triangulation $T$ of $I^i$ so that for any simplex $\Delta\subset T$, there exists $\alpha$ such that $\theta(\hat\sigma(\Delta))\subset U_\alpha$.
Define $\sigma(v):=\theta(\hat\sigma(v))$ for any vertex $v\in T$ and $\sigma|_{\partial I^i}:\equiv p$ $(=\theta(F))$.
Then the contractibility of $\mathcal U$ enables us to define $\sigma$ inductively on the skeleta of $T$ so that
\[\sigma(\partial\Delta)\subset U_\alpha\Longrightarrow\sigma(\Delta)\subset U_\alpha\]
for any $\Delta\subset T$.
In particular $\theta(\hat\sigma(\Delta))\subset U_\alpha\Rightarrow\sigma(\Delta)\subset U_\alpha$.
Since $\diam U_\alpha< r_0$, this implies that $\sigma$ is a projection of $\hat\sigma$.
\end{proof}

\begin{proof}[Well-definedness of ${[\sigma]}$]\renewcommand{\qedsymbol}{}
Let us prove the well-definedness of the map $[\hat\sigma]\mapsto[\sigma]$.
It is obvious from Lemma \ref{lem:hom} that $[\sigma]$ is independent of the choice of a projection $\sigma$ of given $\hat\sigma$.
Furthermore, if there exists a homotopy between two $\hat\sigma$'s, then one can construct a homotopy between some projections of them  by the same argument as in the previous paragraph.
Thus $[\sigma]$ is also independent of the choice of a representative $\hat\sigma$ of $[\hat\sigma]$.
\end{proof}

\begin{proof}[Existence of $\hat\sigma$]\renewcommand{\qedsymbol}{}
Let us show the existence of a lift $\hat\sigma$ for given $\sigma$.
Define $\hat\sigma$ on $I^{i-1}\times\{0\}\cup\partial I^{i-1}\times I$ by the constant map to $\hat p$ and apply Corollary \ref{cor:lift} to it.
Then $\hat\sigma$ extends over $I^i$ so that $|\sigma,\theta\hat\sigma|<r$ for some fixed $r\ll r_0$ since $\mu\ll r_0$.
In particular $\hat\sigma(I^{i-1}\times\{1\})\subset B(\hat p,2r)$.
Then, one can perturb it by using the deformation retraction to $F$ so that $\hat\sigma(I^{i-1}\times\{1\})\subset F$ and $|\sigma,\theta\hat\sigma|<r_0$.
\end{proof}

\begin{proof}[Well-definedness of ${[\hat\sigma]}$]\renewcommand{\qedsymbol}{}
Let us prove the well-definedness of the map $[\sigma]\mapsto[\hat\sigma]$.
Let $H:(I^i,\partial I^i)\times I\to(X,p)$ be a homotopy between $\sigma_0$ and $\sigma_1$, and let $\hat\sigma_0$ and $\hat\sigma_1$ be arbitrary lifts of them.
Define $\hat H:\partial I^{i+1}\setminus(I^{i-1}\times\{1\}\times I)\to M$ by
\[\hat H|_{I^i\times\{0\}}:=\hat\sigma_0,\quad\hat H|_{I^i\times\{1\}}:=\hat\sigma_1,\quad\hat H|_{J^{i-1}\times I}:\equiv\hat p.\]
Then, by Corollary \ref{cor:lift}, $\hat H$ extends over $I^{i+1}$ so that $|H,\theta\hat H|\ll a$ since $r_0\ll a$ (where $K=I^{i-1}\times\{0\}\times I$ and $L=\partial K$).
In particular $\hat H(I^{i-1}\times\{1\}\times I)\subset \hat f^{-1}(I^k(f(p),a))$.
Again one can deform it so that $\hat H(I^{i-1}\times\{1\}\times I)\subset F$, which gives a relative homotopy between $\hat\sigma_0$ and $\hat\sigma_1$.
\end{proof}

Therefore the map $[\sigma]\mapsto[\hat\sigma]$ is well-defined and is a bijection (clearly a homomorphism).
This completes the proof of Theorem \ref{thm:main}(1).
\end{proof}

\section{Generalization to extremal subsets}\label{sec:gen}

In this section, we generalize the results of the previous section to each primitive extremal subset $E$ of the limit space $X$ and prove Theorem \ref{thm:ext}.
In Section \ref{sec:pte}, we extend the notion of a good point of Definition \ref{dfn:pn} to each primitive extremal subset and give a generalization of Theorem \ref{thm:ser} (Theorem \ref{thm:sere}).
In Section \ref{sec:cove}, using it, we prove generalizations of Propositions \ref{prop:shf} and \ref{prop:lift} (Propositions \ref{prop:shfe} and \ref{prop:lifte}), from which Theorem \ref{thm:ext} follows.

As in the previous section, $X$ denotes a fixed $k$-dimensional compact Alexandrov space, where $k<n$, and $M$ denotes a variable $n$-dimensional Alexandrov space that is $\mu$-close to $X$.
We fix $\mu$-approximations $\theta:M\to X$ and $\eta:X\to M$ such that $|\theta\circ\eta,\id_X|<\mu$.
The hat symbol $\hat\ $ is used to indicate natural lifts of objects on $X$ to $M$.
Moreover, in this section, $E$ denotes an $m$-dimensional primitive extremal subset of $X$.

\subsection{Good points of extremal subsets}\label{sec:pte}

In this section, we extend the notion of a good point of Definition \ref{dfn:pn} to each primitive extremal subset and prove a generalization of Theorem \ref{thm:ser}.

Let us first list a few basic properties of regular admissible maps on extremal subsets which will be used many times in this section.
See \cite[\S2]{PP:ext} for the proofs (see also \cite[3.18]{F:reg} for the denseness).
Let $E$ be an $m$-dimensional primitive extremal subset of an Alexandrov space $X$.

\begin{itemize}
\item If an admissible map $f:X\to\mathbb R^l$ is regular at $p\in E$, then $l\le m$ and $f|_E$ is open near $p$.
In particular, if $l=m$, it is a local homeomorphism near $p$.
\item The set of points $p\in E$ at which some admissible map $f:X\to\mathbb R^m$ is regular is open and dense in $E$.
\end{itemize}

Next we recall the characterization of extremal subsets in terms of canonical neighborhoods, essentially proved in \cite{Per:col}.

\begin{lem}\label{lem:s}
Let $S$ be a subset of an Alexandrov space $X$ satisfying the following property:
\begin{itemize}
\item[\rm($\ast$)] for any canonical neighborhood $K$ of $p\in X$ with respect to an admissible map $(f,f_{l+1}):X\to\mathbb R^l\times\mathbb R$, either $f_{l+1}^{-1}(0)\cap K\subset S$ or $f_{l+1}^{-1}(0)\cap K\cap S=\emptyset$ holds.
\end{itemize}
Then the closure of $S$ is an extremal subset of $X$.
\end{lem}

\begin{rem}\label{rem:s}
Conversely, any extremal subset of $X$ satisfies the property ($\ast$).
This follows from the openness of regular admissible maps on extremal subsets mentioned above.
See \cite[\S2]{PP:ext} for the proof.
\end{rem}

In our applications, $S$ is obtained as a set of all points satisfying some condition.
The following are such examples.

\begin{ex}\label{ex:s}
\begin{enumerate}
\item Fix a topological cone $K$ with vertex $o$.
Then, the set of all points whose tangent cones are pointed homeomorphic to $(K,o)$ satisfies the property ($\ast$).
This follows from the stability theorem and the uniqueness of conical neighborhoods (\cite{Kw}).
\item The set of all good points in the sense of Definition \ref{dfn:pn} satisfies the property ($\ast$).
See \cite[2.2]{Per:col} for the proof.
\end{enumerate}
\end{ex}

The proof of Lemma \ref{lem:s} is completely the same as that of \cite[2.3]{Per:col}.
We prove the following claim by reverse induction on $l$.
The lemma immediately follows from the case $l=1$ of the claim.

\begin{clm}[{\cite[2.3]{Per:col}}]\label{clm:s}
Let $f=(f_1,\dots,f_l):X\to\mathbb R^l$ be regular on an open neighborhood $U$ of $p\in X$.
Set $L:=\bigcap_{i>1}f_i^{-1}(f_i(p))$.
Then the restriction of $f_1$ to $\overline{S\cap L}\cap U$ cannot attain its minimum.
\end{clm}

\begin{proof}
The base case $l=\dim X$ is clear since $f$ is a local open homeomorphism on $U$ and the property ($\ast$) implies that $S$ is open and closed in $U$ (note $f_{l+1}\equiv0$).
Suppose the claim does not hold for $l<\dim X$.
We may assume that the minimum is attained at $p$.
We may further assume that $f$ is incomplementable at $p$.
Then, by Theorem \ref{thm:nbhd}, there exist a nonpositive function $f_{l+1}$ with $f_{l+1}(p)=0$ and a canonical neighborhood $K(p,a)$ with respect to $(f,f_{l+1})$.
The minimality of $f_1$ at $p$ implies that $K(p,a)\cap f_{l+1}^{-1}(0)$ is not contained in $S$.
Thus the property ($\ast$) implies that $K(p,a)\cap f_{l+1}^{-1}(0)$ does not intersect $S$.
However, since $p\in\overline{S\cap L}$, there exists $p_1\in S\cap L\cap U_1$, where $U_1:=\mathring K(p,a)\setminus f_{l+1}^{-1}(0)$.
Note that $(f,f_{l+1})$ is regular on $U_1$.
Set $L_1:=\bigcap_{i=2}^{l+1}f_i^{-1}(f_i(p_1))\subset L$.
Then, we have
\[f_1(p)\le f_1(x)<f_1(p)+a\]
for any $x\in\overline{S\cap L_1}\cap U_1$, whereas $f_1(L_1\cap\partial U_1)\subset \{f_1(p)\pm a\}$.
Therefore, the restriction of $f_1$ to $\overline{S\cap L_1}\cap U_1$ can attain its minimum.
This contradicts the induction hypothesis.
\end{proof}

Now we define a good point of a primitive extremal subset.
Recall that $M$ denotes a variable collapsing Alexandrov space that is $\mu$-close to $X$ and that the hat symbol $\hat\ $ indicates natural lifts from $X$ to $M$.
Let $E$ be a primitive extremal subset of $X$, of dimension $m$.

\begin{dfn}\label{dfn:drn}
A point $p\in E$ is said to be \textit{good} in $E$ if it satisfies the following condition (DRN).
\begin{itemize}
\item[(DRN)] For any $R>0$, there exists $\rho=\rho(p,R)>0$ such that for any $q_i\in B(p,\rho)\cap E$ and sufficiently small $r_i>0$ ($i=1,2$), if $\mu$ is small enough, then $M$ contains a \textit{deformation retract neighborhood}, that is, a subset $U$ with $B(\hat p,\rho)\subset U\subset B(\hat p,R)$ that admits deformation retractions onto both $\bar B(\hat q_i,r_i)$.
\end{itemize}
\end{dfn}

In particular, both $\bar B(\hat q_i,r_i)$ are homotopy equivalent.
Although the above definition does not use regular maps as in Definition \ref{dfn:pn}, the essential ideas behind them are the same (the above condition makes sense even when $E=X$).
Taking two points $q_i$ is necessary to prove Lemma \ref{lem:rfe} and Proposition \ref{prop:shfe} below.
Actually it can be replaced with finitely many points in view of the following discussion.

The main theorem of this section is the following generalization of Theorem \ref{thm:ser}.

\begin{thm}\label{thm:sere}
The closure of the set of bad points in $E$, if nonempty, is a proper extremal subset of $E$.
In particular, every point of the main part of $E$ is good.
\end{thm}

In view of Lemma \ref{lem:s}, it suffices to prove the following two lemmas.

\begin{lem}\label{lem:sere1}
Let $p\in E$ be a regular point of an admissible map $f:X\to\mathbb R^m$, where $m=\dim E$.
Then $p$ is a good point of $E$.
\end{lem}

\begin{lem}[cf.\ {\cite[2.2]{Per:col}}]\label{lem:sere2}
Let $K$ be a canonical neighborhood with respect to $(f,f_{l+1}):X\to\mathbb R^{l+1}$ and let $p, q\in\mathring K\cap f_{l+1}^{-1}(0)$.
If $p$ is a good point of $E$, then $q$ is also a good point of $E$.
\end{lem}

Indeed, the first lemma guarantees that the closure of the set of bad points in $E$ is a proper subset of $E$ since the set of such regular points are open and dense in $E$ as mentioned above.
On the other hand, the second lemma means that the set of bad points in $E$ satisfies the property ($\ast$) (see also Remark \ref{rem:s}).
Therefore, Theorem \ref{thm:sere} follows from Lemma \ref{lem:s}.

\begin{proof}[Proof of Lemma \ref{lem:sere1}]
As before, the notation $U\Subset V$ means that the closure of $U$ is contained in the interior of $V$.
Note that $f$ is incomplementable at $p$ since $m=\dim E$.
By Theorem \ref{thm:nbhd}, there exists a function $f_{m+1}$ such that $(f,f_{m+1})$ defines canonical neighborhoods, which we will denote by $K(\cdot,\cdot)$.
For given $R>0$, find $a>0$ such that $K(p,a)\Subset B(p,R)$ and define $\rho=\rho(p,R)>0$ so that $B(p,\rho)\Subset K(p,a)$.
Let $q_i\in B(p,\rho)\cap E$ ($i=1,2$).
Note that $q_i\in K(p,a)\cap f_{m+1}^{-1}(0)$ since $m=\dim E$.
Take $\tilde r_i>0$ such that $\dist_{q_i}$ is regular on $\bar B(q_i,\tilde r_i)\setminus\{q_i\}$ and $B(q_i,\tilde r_i)\Subset B(p,\rho)$.
Let $b_i>r_i>0$ be so small that $B(q_i,r_i)\Subset K(q_i,b_i)\Subset B(q_i,\tilde r_i)$.
Then, if $\mu$ is small enough, the lift $\hat K(p,a)$ is a deformation retract neighborhood for $\bar B(\hat q_i,r_i)$ such that $B(\hat p,\rho)\subset\hat K(p,a)\subset B(\hat p,R)$.
Indeed, we may assume that $\hat f$ is regular on $\hat K(p,a)$ and in addition that $(\hat f,\hat f_{m+1})$ is regular if $\hat f_{m+1}\le-b_i$.
Then it easily follows from the fibration theorem \ref{thm:adm} that $\hat K(q_i,b_i)$ is a deformation retract of $\hat K(p,a)$.
Similarly $\bar B(\hat q_i,r_i)$ is a deformation retract of $B(\hat q_i,\tilde r_i)$.
Since $B(\hat q_i,r_i)\Subset\hat K(q_i,b_i)\Subset B(\hat q_i,\tilde r_i)\Subset\hat K(p,a)$ if $\mu$ is small enough, $\bar B(\hat q_i,r_i)$ is a deformation retract of $\hat K(p,a)$.
\end{proof}

\begin{proof}[Proof of Lemma \ref{lem:sere2}]
The proof is similar to that of \cite[2.2]{Per:col}.
First of all, notice that $p\in E$ implies $q\in E$ by Remark \ref{rem:s}.

Let us first define $\rho(q,R)$ for given $R>0$.
We denote by $K(\cdot,\cdot)$ canonical neighborhoods with respect to $(f,f_{l+1})$.
Let $a>0$ be so small that $K(q,a)\Subset K\cap B(q,R)$ and $K(p,a)\Subset K$.
Find $r>0$ such that $B(p,r)\Subset K(p,a)$ and take $\rho(p,r)$ by the condition (DRN) for $p$.
Choose $b>0$ so that $K(p,b)\Subset B(p,\rho(p,r))$.
Finally, define $\rho(q,R)>0$ so that $B(q,\rho(q,R))\Subset K(q,b)$.

Recall that the fibration theorem \ref{thm:adm} yields a homeomorphism $\Phi:I^l\times F\to K$ respecting $(f,f_{l+1})$, where $I^l$ is some $l$-dimensional closed cube in $\mathbb R^l$ and $F$ is a fiber of $f|_K$ (see Remark \ref{rem:nbhd}).
Define a translation $\tau:K(p,a)\to K(q,a)$ respecting $f_{l+1}$ by $\tau:=\Phi\circ(T_v,\id_F)\circ\Phi^{-1}$, where $T_v$ denotes the translation on $\mathbb R^l$ by a vector $v=f(q)-f(p)$.
Similarly, if $\mu\ll b$ is small enough, there exists a homeomorphism $\hat\Phi:I^l\times\hat F\to\hat K$ respecting $\hat f$ and in addition respecting $\hat f_{l+1}$ if $\hat f_{l+1}\le-b$, where $\hat F$ is a fiber of $\hat f|_{\hat K}$.
Define a translation $\hat\tau:\hat K(p,a)\to\hat K(q,a)$ respecting $\hat f_{l+1}$ if $\hat f_{l+1}\le-b$ in the same way as above.
Note that $\hat\Phi$ and $\hat\tau$ are not natural lifts of $\Phi$ and $\tau$ in any sense.

Let $y_i\in B(q,\rho(q,R))\cap E$ ($i=1,2$).
Take $\tilde s_i>0$ such that $\dist_{y_i}$ is regular on $\bar B(y_i,\tilde s_i)\setminus\{y_i\}$ and $B(y_i,\tilde s_i)\Subset B(q,\rho(q,r))$.
Put $x_i:=\tau^{-1}(y_i)$.
Then $x_i\in B(p,\rho(p,r))\cap E$ since $\tau$ sends $K(p,b)$ to $K(q,b)$ and in addition it can be chosen to respects $E$.
The latter follows from the relative fibration theorem for extremal subsets (see \cite[\S9]{K:stab}).
Take sufficiently small $r_i>0$ for each $x_i$ by the condition (DRN) for $p$ such that $B(x_i,r_i)\Subset\tau^{-1}(B(y_i,\tilde s_i))$ and let $s_i>0$ be so small that $B(y_i,s_i)\Subset\tau(B(x_i,r_i))$.
Then, if $\mu$ is small enough, there exists a deformation retract neighborhood $U$ for $\bar B(\hat x_i,r_i)$ with $B(\hat p,\rho(p,r))\subset U\subset B(\hat p,r)$.

Let us prove that $\hat\tau(U)$ is a deformation retract neighborhood for $\bar B(\hat y_i,s_i)$ such that $B(\hat q,\rho(q,R))\subset\hat\tau(U)\subset B(\hat q,R)$.
Indeed the inclusions hold since $\hat\tau$ sends $\hat K(p,a)$ (resp.\ $\hat K(p,b)$) to $\hat K(q,a)$ (resp.\ $\hat K(q,b)$).
Let us show that $\bar B(\hat y_i,s_i)$ is a deformation retract of $\hat\tau(U)$.
Recall that $\hat\tau(\bar B(\hat x_i,r_i))$ is a deformation retract of $\hat\tau(U)$.
On the other hand, the fibration theorem implies that $\bar B(\hat y_i,s_i)$ is a deformation retract of $B(\hat y_i,\tilde s_i)$.
Furthermore, we may assume $B(\hat y_i,s_i)\Subset\hat\tau(B(\hat x_i,r_i))\Subset B(\hat y_i,\tilde s_i)\Subset\tau(U)$ by the following claim (this is not trivial since $\hat\tau$ is not a lift of $\tau$).
This completes the proof.
\end{proof}

\begin{clm}\label{clm:sere}
In the above proof, for given $y_i\in\mathring K(q,a)$ and $\tilde s_i>0$ ($i=1,2$), one can choose homeomorphisms $\Phi:I^l\times F\to K$ and $\hat\Phi:I^l\times\hat F\to\hat K$ and radii $r_i>s_i>0$ so that the translations $\tau$ and $\hat\tau$ satisfy
\begin{gather*}
B(y_i,s_i)\Subset\tau(B(x_i,r_i))\Subset B(y_i,\tilde s_i),\\
B(\hat y_i,s_i)\Subset\hat\tau(B(\hat x_i,r_i))\Subset B(\hat y_i,\tilde s_i),
\end{gather*}
where $x_i=\tau^{-1}(y_i)$.
Note that the choices of $r_i$ and $s_i$ depend only on $X$ but not on $M$.
\end{clm}

\begin{proof}
The idea of the proof is actually simple.
If $f_{l+1}(y_i)=0$, then one can estimate the sizes of the balls by those of canonical neighborhoods with respect to $(f,f_{l+1})$, as we have done for $p$ and $q$ in the proof of Lemma \ref{lem:sere2}.
Since the translations do not change the sizes of canonical neighborhoods, this gives the desired conclusion.
If $f_{l+1}(y_i)\neq0$, then one can use reverse induction on $l$ to reduce the problem to the former case.
However, in the latter case, we need a more complicated argument.

For simplicity, we assume $v=f(q)-f(p)=(v_1,0,\dots,0)$ and $v_1>0$ (the general case is similar).
Let us first explain how to construct such homeomorphisms when only one point $y\in\mathring K(q,a)$ and $\tilde s>0$ are given.
The two-point case will be dealt with later.

We first define $x=\tau^{-1}(y)\in\mathring K(p,a)$ by reverse induction on $l$ as follows.
\begin{enumerate}
\item If $f_{l+1}(y)=0$, then define $x$ as the unique point such that $f(x)=f(y)-v$ and $f_{l+1}(x)=0$.
\item If $f_{l+1}(y)\neq0$, then cover the level set $f_{l+1}^{-1}(f_{l+1}(y))\cap K$ by finitely many canonical neighborhoods respecting $(f,f_{l+1})$, using Theorem \ref{thm:nbhd}.
Then one can define $x$ by reverse induction on $l$.
More specifically, the induction shows that there exists a sequence of canonical neighborhoods $K_i$ ($1\le i \le N$) with respect to admissible maps $(f^i,f^i_{l_i+1}):X\to\mathbb R^{l_i+1}$ ($l_i>l$) whose first $(l+1)$-coordinates coincide with $(f,f_{l+1})$ and points $x_i\in\mathring K_i\cap(f^i_{l_i+1})^{-1}(0)$ such that
\begin{align*}
&x_N=y, & &\\
&x_i=\tau_{i+1}^{-1}(x_{i+1}), & &f_1(x_i)<f_1(x_{i+1})\\
&x=\tau_1^{-1}(x_1), & &f_1(x)=f_1(y)-v_1,
\end{align*}
where $\tau_{i+1}^{-1}(x_{i+1})$ is naturally defined on $K_{i+1}$ as in the case (1).
Note that $x_i$ is also contained in $\mathring K_{i+1}\cap(f^{i+1}_{l_{i+1}+1})^{-1}(0)$.
\end{enumerate}

In the case (1), the claim follows by an argument using canonical neighborhoods of $x$ and $y$ with respect to $(f,f_{l+1})$, as in the proof of Lemma \ref{lem:sere2}.

Let us consider the case (2).
We first construct the desired homeomorphism $\Phi$.
The construction of $\hat\Phi$ is carried out in parallel with it as we will see later.
Let $\Phi_i:I^{l_i}\times F_i\to K_i$ be a homeomorphism respecting $(f^i,f^i_{l_i+1})$, where $I^{l_i}$ is some $l_i$-dimensional cube and $F_i$ is a fiber of $f^i|_{K_i}$.
We first glue them along the fibers of $f_1$ to construct part of $\Phi$.
Let $K_i(\cdot,\cdot)$ denote canonical neighborhoods with respect to $(f^i,f^i_{l_i+1})$.
For given $\tilde s>0$, find $a_N>0$ such that $K_N(y,a_N)\Subset B(y,\tilde s)\cap K(q,a)$.
Then choose smaller and smaller $a_i>0$ ($1\le i\le N-1$) so that
\[K_i(x_i,a_i)\Subset K_{i+1}(x_i,a_{i+1}).\]
In particular, note that $K_1(x,a_1)\Subset K(p,a)$ since every $K_i$ respects $(f,f_{l+1})$.
Let $I_i^{l_i-1}(a_i)$ denote the closed $a_i$-neighborhood of $(f^i_2(x_i),\dots f^i_{l_i}(x_i))$ in $\mathbb R^{l_i-1}$ with respect to the maximum norm and let $F_i(a_i)$ denote the superlevel set $\{f_{l_i+1}\ge-a_i\}$ in $F_i$.
Then, since $K_1(x_1,a_1)\Subset K_2(x_1,a_2)$, we can define an embedding $\Psi_{21}:I_1^{l_1-1}(a_1)\times F_1(a_1)\to I_2^{l_2-1}(a_2)\times F_2(a_2)$ respecting the first $l$-coordinates by
\[\Psi_{21}:=\Phi_2^{-1}\circ\Phi_1|_{\{f_1(x_1)\}\times I_1^{l_1-1}(a_1)\times F_1(a_1)}.\]
Set $I_2:=[f_1(x_{1}),f_1(x_2)]$ and define an embedding $\Phi_2':I_2\times I_1^{l_1-1}(a_1)\times F_1(a_1)\to K$ respecting $(f,f_{l+1})$ by
\[\Phi_2':=\Phi_2\circ(\pr_1,\Psi_{21}\circ\pr_{>1}),\]
where  $\pr_1$ and $\pr_{>1}$ denote the projections onto the first and other factors, respectively.
Note that $\Phi_2'$ coincides with $\Phi_1$ on $\{f_1(x_1)\}\times I_1^{l_1-1}(a_1)\times F_1(a_1)$.
We can repeat such a procedure to obtain embeddings $\Phi_i':I_i\times I_1^{l_i-1}(a_1)\times F_1(a_1)\to K$ respecting $(f,f_{l+1})$, where $\Phi_1'=\Phi_1$ and
\begin{equation*}
I_i=
\begin{cases}
\hfil [f_1(x)-a_1,f_1(x_1)] & \hfil i=1 \\
\hfil [f_1(x_{i-1}),f_1(x_i)] & 2\le i\le N-1 \\
[f_1(x_{N-1}),f_1(y)+a_N] & \hfil i=N,
\end{cases}
\end{equation*}
which coincide with each other on the boundaries of $I_i$.
Finally, we set $\Phi:=\Phi_i'$ on each $I_i\times I_1^{l_1-1}(a_1)\times F_1(a_1)$.

We may assume that $\Phi$ is actually defined on a slightly larger domain than the above one.
Then, using Siebenmann's result \cite[6.9]{Si} (see also \cite[\S1 Complement to Theorem A]{Per:alex}), one can extend $\Phi$ to a homeomorphism $I^l\times F\to K$ respecting $(f,f_{l+1})$.

Let $\tau:K(p,a)\to K(q,a)$ be the translation with respect to $\Phi$.
We estimate radii $r>s>0$ such that $B(y,s)\Subset\tau(B(x,r))\Subset B(y,\tilde s)$ in terms of sizes of canonical neighborhoods.
Let $r>0$ be such that $B(x,r)\Subset K_1(x,a_1)$.
Then the choices of $a_i$ imply that $\tau(B(x,r))\Subset\tau(K_1(x,a_1))\subset K_N(y,a_N)\Subset B(y,\tilde s)$.
Similarly we estimate $s>0$ such that $B(y,s)\Subset\tau(B(x,r))$ as follows.
Find $b_1>0$ such that $K_1(x,b_1)\Subset B(x,r)$ and choose smaller and smaller $b_i>0$ ($2\le i\le N$) so that
\[K_i(x_{i-1},b_i)\Subset K_{i-1}(x_{i-1},b_{i-1}).\]
Finally let $s>0$ be such that $B(y,s)\Subset K_N(y,b_N)$.
Then it satisfies the desired inclusion.
Indeed, since $K_1(x_1,b_1)\Supset K_2(x_1,b_2)$, we have
\[\Psi_{21}(I_1^{l_1-1}(b_1)\times F_1(b_1))\supset I_2^{l_2-1}(b_2)\times F_2(b_2),\]
and hence
\[\Phi_2'(\{f_1(x_2)\}\times I_1^{l_i-1}(b_1)\times F_1(b_1))\supset f_1^{-1}(f_1(x_2))\cap K_2(x_2,b_2).\]
We can repeat such an argument to show
\[\Phi_N'(I(f_1(y),b_1)\times I_1^{l_i-1}(b_1)\times F_1(b_1))\supset K_N(y,b_N).\]
This implies $\tau(B(x,r))\Supset\tau(K_1(x,b_1))\supset K_N(y,b_N)\Supset B(y,s)$.

Now, suppose $\mu\ll b_N$ is small enough and lift all the canonical neighborhoods discussed above to $M$.
Then, one can construct the desired homeomorphism $\hat\Phi$ in parallel with the construction of $\Phi$ as follows.
Let $\hat\Phi_i:I^{l_i}\times\hat F_i\to\hat K_i$ be a homeomorphism respecting $\hat f^i$, where $\hat F_i$ is a fiber of $\hat f^i|_{\hat K_i}$.
In addition, we may assume that $\hat\Phi_i$ respects $\hat f^i_{l_i+1}$ if $\hat f^i_{l_i+1}\le -b_N$.
Since the inclusions of canonical neighborhoods in $X$ hold for their lifts in $M$, one can glue $\hat\Phi_i$ and extend it to obtain $\hat\Phi$ in the same way as above.
Furthermore, the translation $\hat\tau:\hat K(p,a)\to\hat K(q,a)$ with respect to $\hat\Phi$ satisfies
\[B(\hat y,s)\Subset\hat\tau(B(\hat x,r))\Subset B(\hat y,\tilde s)\]
since the above estimates using canonical neighborhoods still hold in $M$.
This completes the proof of the one-point case.

Finally, we outline the proof of the two-point case.
First recall the construction of $x=\tau^{-1}(y)$ at the beginning of the above proof.
We define the locus $T$ of this translation as follows: (1) if $f_{l+1}(y)=0$, then define it as the inverse image of the shortest path $f(x)f(y)$ under $f|_{f_{l+1}^{-1}(0)\cap K}$; (2) if $f_{l+1}(y)\neq 0$, then define it by reverse induction on $l$ as before.
Suppose two points $y_i\in\mathring K(q,a)$ ($i=1,2$) are given (we may assume $f_{l+1}(y_i)\neq 0$).
Let $T_i$ denote the locus through $y_i$.
We may assume that $T_1$ is defined beyond $x_1$ and $y_1$ to the boundary of $K$.
If $y_2\in T_1$, one can construct the desired homeomorphism in the same way as the one-point case.
On the other hand, if $y_2\notin T_1$, one can define $T_2$ so that $T_1\cap T_2\neq 0$.
Then we can construct homeomorphisms on disjoint neighborhoods of (lifts of) $T_i$ in the same way as the one-point case and use Siebenmann's result \cite[6.10]{Si} to obtain the desired homeomorphisms.
This completes the proof.
\end{proof}

We conclude this section by defining regular fibers over primitive extremal subsets (cf.\ Section \ref{sec:ser}).
Let $E$ be a primitive extremal subset of $X$ and $\mathring E$ its main part.
Let $p\in\mathring E$ and take $r>0$ such that $\dist_p$ is regular on $\bar B(p,r)\setminus\{p\}$.
We call such a pair $(p,r)$ \textit{fiber data} on $E$.
Suppose $\mu\ll r$ is small enough and lift it to $M$.
We call the closed ball $\bar B(\hat p,r)$ a \textit{regular fiber} over $E$ in $M$.

As before, all regular fibers over $E$ are homotopy equivalent (cf.\ Lemma \ref{lem:rf}).
To show this, we need the following basic fact.

\begin{lem}\label{lem:conn}
The main part $\mathring E$ is connected.
\end{lem}

\begin{proof}
Note that $F:=E\setminus\mathring E$ is also an extremal subset (of lower dimension).
Assume $\mathring E$ can be written as a union of disjoint nonempty open subsets $U_1$ and $U_2$.
We show that $U_1\cup F$ is an extremal subset, which contradicts the primitiveness of $E$.
Indeed, $F$ is invariant under any gradient flow.
Suppose that $p_1\in U_1$ and $p_2\in U_2$ can be joined by some gradient curve $\gamma$.
Then $\gamma$ is contained in $\mathring E=E\setminus F$ since so are its endpoints.
This contradicts the assumption.
Therefore $U_1\cup F$ is invariant under any gradient flow.
\end{proof}

\begin{lem}\label{lem:rfe}
Let $(p,r)$ and $(q,s)$ be fiber data on $E$.
Then, if $\mu$ is small enough, the corresponding regular fibers $\bar B(\hat p,r)$ and $\bar B(\hat q,s)$ in $M$ are homotopy equivalent.
\end{lem}

\begin{proof}
If $p=q$, then the claim immediately follows from the fibration theorem.
Assume $p\neq q$ and connect them by a curve $\gamma$ in $\mathring E$.
Then, every point on $\gamma$ is good in $E$ by Theorem \ref{thm:sere}.
Fix arbitrary $R>0$ and set $\rho:=\inf_{x\in\gamma}\rho(x,R)>0$.
Choose points $p=x_1$, $x_2$, \dots, $x_N=q$ on $\gamma$ in this order so that one can find $y_i\in\gamma\cap B(x_i,\rho)\cap B(x_{i+1},\rho)$ ($1\le i\le N-1$).
Put $y_0:=p$ and $y_N:=q$.
Take sufficiently small $t>0$ for all $y_i$ by the condition (DRN).
Then, if $\mu$ is small enough, there exist deformation retract neighborhoods $U_i$ for $\bar B(\hat y_{i-1},t)$ and $\bar B(\hat y_i,t)$ ($1\le i\le N$).
Thus $\bar B(\hat p,t)$ is homotopic to $\bar B(\hat q,t)$.
Together with the case $p=q$, this completes the proof.
\end{proof}

\begin{rem}\label{rem:rfe}
Let $f:X\to\mathbb R^m$ be regular at $p\in E$, where $m=\dim E$, as in Lemma \ref{lem:sere1}.
Note that $p\in\mathring E$ since $E\setminus\mathring E$ is an extremal subset of lower dimension.
Take a canonical neighborhood $K(p,a)$ respecting $f$ as in the proof of the lemma.
Then, the fibration theorem shows that the lift $\hat K(p,a)$ is homotopy equivalent to some regular fiber $\bar B(\hat p,r)$.
Note that $\hat K(p,a)$ also admits a deformation retraction onto a fiber of $\hat f|_{\hat K(p,a)}$.
In particular, in case $E=X$, the above regular fiber is homotopy equivalent to the original one in Section \ref{sec:ser}.
\end{rem}

\subsection{Good covering and extremal subsets}\label{sec:cove}

In this section, we study the relation between a lift of a good covering of $X$ and extremal subsets of $X$ and prove Theorem \ref{thm:ext}.

Let $\mathcal U=\{U_\alpha\}_{\alpha=1}^N$ be a good covering of $X$.
We use the same notation as in Section \ref{sec:prf}: $h_\alpha$ denotes a strictly concave function defining $U_\alpha$ and $p_\alpha$ denotes its unique maximum point.
Moreover, for a subset $A\subset\{1,\dots,N\}$, we define $U_A$, $h_A$, and $p_A$ as before.
Natural lifts of them in $M$ are denoted by $\hat U_\alpha$, $\hat h_\alpha$, $\hat U_A$, and $\hat h_A$.
In addition, the following notation is used here.
\[E_A:=\bigcap_{E\cap U_A\neq\emptyset} E=\bigcap_{p_A\in E}E,\]
where $E$ runs over all extremal subsets of $X$ satisfying the above conditions.
The second equality follows from the fact that the contraction of $U_A$ to $p_A$ given by gradient flows preserves extremal subsets.
Note that $E_A$ is a primitive extremal subset and that $p_A$ is in the main part $\mathring E_A$.

The following is a generalization of Proposition \ref{prop:shf}.

\begin{prop}\label{prop:shfe}
Let $\mathcal U=\{U_\alpha\}_{\alpha=1}^N$ be a good covering of $X$ as above.
Suppose $\mu$ is small enough.
Then, for any $A$, the lift $\hat U_A$ has the homotopy type of a regular fiber over $E_A$.
Furthermore, if $A\subset A'$ and $E_A=E_{A'}$, then the inclusion $\hat U_{A'}\hookrightarrow\hat U_A$ is a homotopy equivalence.
\end{prop}

\begin{proof}
First we show that $\hat U_A$ has the homotopy type of a regular fiber over $E_A$.
Since $p_A\in\mathring E_A$, it can define a regular fiber over $E_A$.
Take $\tilde r>0$ such that $\dist_{p_A}$ is regular on $\bar B(p_A,\tilde r)\setminus\{p_A\}$ and $B(p_A,\tilde r)\Subset U_A$.
Choose $\varepsilon>0$ so that $U_A(\varepsilon):=\{h_A>h_A(p_A)-\varepsilon\}\Subset B(p_A,\tilde r)$.
Find $r>0$ such that $B(p_A,r)\Subset U_A(\varepsilon)$.
Now suppose $\mu$ is small enough and lift them to $M$.
Then we have the inclusions
\[\bar B(\hat p_A,r)\subset\Cl(\hat U_A(\varepsilon))\subset B(\hat p_A,\tilde r)\subset\hat U_A,\]
where $\Cl(\hat U_A(\varepsilon))$ denotes the closure of a natural lift of $U_A(\varepsilon)$.
Here $\bar B(\hat p_A,r)$ is a deformation retract of $B(\hat p_A,\tilde r)$ by the fibration theorem.
Furthermore $\Cl(\hat U_A(\varepsilon))$ is a deformation retract of $\hat U_A$ by the gradient flow of $\hat h_A$.
Therefore $\bar B(\hat p_A,r)$ is also a deformation retract of $\hat U_A$.

Next we show that $\hat U_{A'}\hookrightarrow\hat U_A$ is a homotopy equivalence if $E_A=E_{A'}=:E$.
Take regular fibers $\bar B(p_A,r)$ and $\bar B(p_{A'},r)$ as above such that the inclusions $\bar B(p_A,r)\hookrightarrow\hat U_A$ and $\bar B(p_{A'},r)\hookrightarrow\hat U_{A'}$ are homotopy equivalences.
Let $\Phi:\bar B(p_{A'},r)\to\bar B(p_A,r)$ be a homotopy equivalence constructed in the proof of Lemma \ref{lem:rfe}.
Recall that the construction is done on an arbitrarily small neighborhood of a curve $\gamma$ connecting $p_{A'}$ and $p_A$ in $\mathring E$.
On the other hand, using the contractibility of $U_A$, one can connect $p_{A'}$ to $p_A$ by a (broken) gradient curve in $U_A$.
This gradient curve is contained in $\mathring E$ since so are its endpoints.
Thus we can use it as $\gamma$.
In particular we can take $R>0$ in the proof of Lemma \ref{lem:rfe} so that $U(\gamma,R)\Subset U_A$.
Then one can easily check that the composition $\bar B(p_{A'},r)\xrightarrow\Phi\bar B(p_A,r)\hookrightarrow\hat U_A$ is homotopic to the inclusion $\bar B(p_{A'},r)\hookrightarrow\hat U_A$ in $\hat U_A$.
Hence $\bar B(p_{A'},r)\hookrightarrow\hat U_A$ is also a homotopy equivalence and so is $\hat U_{A'}\hookrightarrow\hat U_A$.
\end{proof}

Let $E$ be a primitive extremal subset of $X$.
To prove Theorem \ref{thm:ext}(2), we need the following additional observation (the primitiveness is not needed here).

\begin{lem}\label{lem:sc}
Let $p\in E$ and let $h$ be a strictly concave function on a small neighborhood $U$ of $p$ constructed in Theorem \ref{thm:sc}.
Then $h'(\uparrow_x^E)>c_p$ for any $x\in U\setminus E$, where $c_p$ is a positive constant (depending only on the dimension and the volume of $\Sigma_p$).
In particular, the distance function $\dist_E$ from $E$ is monotonically decreasing along the gradient flow of $h$.
\end{lem}

\begin{proof}
Let us first recall the construction of $h$ (see \cite[\S4]{K:reg} for more details).
Take sufficiently small $\varepsilon>0$ which will be determined later by the volume of $\Sigma_p$.
Let $R>0$ be so small that the rescaled ball $R^{-1}B(p,R)$ is sufficiently close to the unit ball in $T_p$ (depending on $\varepsilon$).
Take a $\pi R/16$-net $\{q_\alpha\}_{\alpha=1}^N\subset\partial B(p,R)$ and a maximal $\varepsilon R$-discrete set $\{q_{\alpha\beta}\}_{\beta=1}^{N_\alpha}\subset\partial B(p,R)\cap B(q_\alpha,\pi R/16)$ for each $\alpha$.
Note that the Bishop-Gromov inequality implies $N_\alpha\ge c\vol\Sigma_p/\varepsilon^{k-1}$, where $c$ is a constant depending only on $k=\dim X$.
Let $r>0$ be small enough compared to $R$ and $\varepsilon$.
Define functions $h_\alpha$ and $h$ on $B(p,r)$ by
\[h_\alpha:=\frac1{N_\alpha}\sum_{\beta=1}^{N_\alpha}\varphi_{R,r}\circ\dist_{q_{\alpha\beta}},\quad h:=\min_{1\le\alpha\le N}h_{\alpha}\]
where $\varphi_{R,r}(t)=(t-R)-(t-R)^2/4r$.
Then $h_\alpha$ is strictly concave on $B(p,r)$ since $\varphi_{R,r}''(t)=-1/2r$ and $\dist_{q_{\alpha\beta}}'(\xi)$ for fixed $\xi$ are uniformly bounded away from zero for a majority of $\beta$ (see the argument below).
Therefore $h$ is also strictly concave.
Furthermore, $h$ has a strict maximum at $p$ since the directions $\uparrow_p^{q_\alpha}$ ($1\le\alpha\le N$) are $\pi/8$-dense in $\Sigma_p$.

Let us show that $h'(\uparrow_x^E)>c_p$ for any $x\in B(p,r)\setminus E$.
Let $y\in E$ be a closest point to $x$.
Note that $|\tilde\angle q_{\alpha\beta}xy+\tilde\angle q_{\alpha\beta}yx-\pi|<\varepsilon$ as $r\ll R$.
Since $\tilde\angle q_{\alpha\beta}yx\le\pi/2$ by the extremality of $E$, we have $\tilde\angle q_{\alpha\beta}xy\ge\pi/2-\varepsilon$.
On the other hand, the directions $\uparrow_x^{q_{\alpha\beta}}$ ($1\le\beta\le N_\alpha$) are $\varepsilon/2$-discrete for each $\alpha$.
Hence the number of $\beta$ such that $||\uparrow_x^E\Uparrow_x^{q_{\alpha\beta}}|-\pi/2|<\varepsilon$ is less than $C/\varepsilon^{k-2}$, where $C$ is a constant depending only on $k$.
In particular it is much less than $N_\alpha$ provided $\varepsilon$ is small enough.
Therefore, the mean value of $\varphi_{R,r}'(|q_{\alpha\beta}x|)\cos|\uparrow_x^E\Uparrow_x^{q_{\alpha\beta}}|$ is less than some negative constant, in other words, $h_\alpha'(\uparrow_x^E)>c_p$.
Thus we obtain $h'(\uparrow_x^E)>c_p$.
Furthermore, since $h'(\uparrow_x^E)\le\langle\nabla_xh,\uparrow_x^E\rangle\le-\dist_E'(\nabla_x h)$ by the property of the gradient (see Section \ref{sec:good}), $\dist_E$ is monotonically decreasing along the gradient flow of $h$.
\end{proof}

From now on, we assume that $E$ has no proper extremal subsets.
Let $\mathcal U=\{U_\alpha\}_{\alpha=1}^N$ be a good covering of $E$ in $X$ (we assume $U_\alpha\cap E\neq\emptyset$ for any $\alpha$ whenever we take such a covering).
Note that the restriction $\{U_\alpha\cap E\}_{\alpha=1}^N$ is also good in the topological sense, that is, any nonempty intersection of them is contractible, since gradient flows preserve extremal subsets.
Furthermore, the above lemma implies that $U_A\neq\emptyset\Leftrightarrow U_A\cap E\neq\emptyset$ for any $A\subset\{1,\dots,N\}$ (otherwise $h_A$ is regular at $p_A$).
In particular, the assumption that $E$ has no proper extremal subsets means $E_A=E$ for any $A$.

Now we can prove Theorem \ref{thm:ext}(2).

\begin{proof}[Proof of Theorem \ref{thm:ext}(2)]
Take a good covering $\mathcal U=\{U_\alpha\}_{\alpha=1}^N$ of $E$ as above.
Let $\rho>0$ be so small that $U(E,\rho)\Subset\bigcup_{\alpha=1}^NU_\alpha$.
Set $V_\alpha:=U_\alpha\cap U(E,\rho)$ and $V_A:=\bigcap_{\alpha\in A}V_\alpha$ and let $\hat V_\alpha$ and $\hat V_A$ denote natural lifts of them, respectively.
Then the same argument as in the proof of Proposition \ref{prop:shfe} show that $\hat V_A$ has the homotopy type of a regular fiber over $E$ and that $\hat V_{A'}\hookrightarrow\hat V_A$ is a homotopy equivalence for any $A\subset A'$.
Indeed, Lemma \ref{lem:sc} implies that $h_A'(\uparrow_x^E)>c_A>0$ for any $x\in V_A\setminus E$.
Together with $h_A'(\uparrow_x^E)\le-\dist_E'(\nabla _xh_A)$, this shows that the gradient flow of $h_A$ crushes $V_A$ to $p_A$ in $V_A$.
Similarly, since $\hat h_A'(\uparrow_{\hat x}^{\hat E})>c_A$ for any $\hat x\in\hat V_A\setminus U(\hat E,\rho/2)$ by the lower semicontinuity, the gradient flow of $\hat h_A$ pushes $\hat V_A$ into a neighborhood of $\hat p_A$ in $\hat V_A$.
Thus, one can replace $U_A$ and $\hat U_A$ with $V_A$ and $\hat V_A$ respectively in the proof of Proposition \ref{prop:shfe} to obtain the same statement for $\hat V_A$.
Furthermore, this flow argument shows that $\{\hat V_\alpha\}_{\alpha=1}^N$ has the same nerve as the restricted cover $\{U_\alpha\cap E\}_{\alpha=1}^N$.
Therefore, $H^\ast(\hat V_A)$ defines a locally constant presheaf on a (topological) good cover of $E$, which enables us to compute $H^\ast(U(\hat E,\rho))$ by a standard argument of spectral sequences.
\end{proof}

Next we prove Theorem \ref{thm:ext}(1).
Proposition \ref{prop:shfe} implies the homotopy lifting property over $E$ with respect to a good covering (cf.\ Proposition \ref{prop:lift}).
The proof is omitted since it is almost the same as that of Proposition \ref{prop:lift}, except that Proposition \ref{prop:shfe} is used instead of Proposition \ref{prop:shf}.

\begin{prop}\label{prop:lifte}
Let $\mathcal U=\{U_\alpha\}_{\alpha=1}^N$ be a good covering of $E$.
Suppose $E$ has no proper extremal subsets and $\mu$ is small enough.
Let $K$ be a finite simplicial complex.
Suppose $\sigma:K\times I\to E$ and $\hat\sigma:K\times\{0\}\to M$ satisfy that
\begin{itemize}
\item for any simplex $\Delta\subset K$, there exists $\alpha$ such that $\sigma(\Delta\times I)\subset U_\alpha$;
\item for any $\Delta\subset K$, if $\sigma(\Delta\times\{0\})\subset U_\alpha$, then $\hat\sigma(\Delta\times\{0\})\subset\hat U_\alpha$.
\end{itemize}
Then, $\hat\sigma$ can be extended to a map on $K\times I$ so that
\begin{itemize}
\item for any $\Delta\subset K$, if $\sigma(\Delta\times I)\subset U_\alpha$, then $\hat\sigma(\Delta\times I)\subset\hat U_\alpha$;
\item for any $\Delta\subset K$, if $\sigma(\Delta\times\{1\})\subset U_\alpha$, then $\hat\sigma(\Delta\times\{1\})\subset\hat U_\alpha$.
\end{itemize}
Furthermore, if $L$ is a subcomplex of $K$ and $\hat\sigma$ is already defined on $L\times I$ so that the above conclusions hold for any $\Delta\subset L$, then we can extend it.
\end{prop}

The above proposition immediately implies the homotopy lifting property over $E$ with respect to the approximation $\theta:M\to X$ (cf.\ Corollary \ref{cor:lift}).
The proof is the same as that of Corollary \ref{cor:lift} and is omitted.

\begin{cor}\label{cor:lifte}
Suppose $E$ has no proper extremal subsets.
Then, for any $R>0$, there exists $r>0$ such that the following holds provided $\mu$ is small enough:
Let $K$ be a finite simplicial complex and $L$ a subcomplex.
Suppose $\sigma:K\times I\to E$ and $\hat \sigma:K\times\{0\}\cup L\times I\to M$ satisfy $|\sigma,\theta\hat \sigma|<r$ on $K\times\{0\}\cup L\times I$.
Then, $\hat\sigma$ can be extended to a map on $K\times I$ so that $|\sigma,\theta\hat\sigma|<R$ on $K\times I$.
\end{cor}

As before, we need the following lemma (cf.\ Lemma \ref{lem:hom}).
The proof is the same as that of Lemma \ref{lem:hom} since the contractions of a good covering given by gradient flows preserve extremal subsets.

\begin{lem}\label{lem:home}
There exists a constant $R_0>0$ such that for any finite simplicial complex $K$ and its subcomplex $L$, any two maps $\sigma_i:(K,L)\to(E,p)$ ($i=0,1$) that are $R_0$-close to each other are homotopic in $E$ relative to $L$.
\end{lem}

We also need the following additional lemma for the proof of Theorem \ref{thm:ext}(1).
The proof is essentially included in that of Lemma \ref{lem:sc}.

\begin{lem}[{\cite[3.1(2)]{PP:ext}}]\label{lem:dist}
There exists a constant $\varepsilon>0$ such that the distance function $\dist_E$ is $\varepsilon$-regular on $\bar U(E,\varepsilon)\setminus E$, i.e.\ $|\nabla\dist_E|>\varepsilon$ (the choice of $\varepsilon$ depends only on $k$, $\kappa$, an upper bound on the diameter of $X$, and a lower bound on the volume of $X$).
\end{lem}

We are now in a position to prove Theorem \ref{thm:ext}(1).
The proof is almost the same as that of Theorem \ref{thm:main}(1) except that Lemma \ref{lem:dist} is used.

\begin{proof}[Proof of Theorem \ref{thm:ext}(2)]

Let us first define some constants determined by $X$.
Fix fiber data $(p,\tilde r)$ on $E$.
Let $\varepsilon>0$ and $R_0>0$ be the constants of Lemma \ref{lem:dist} and Lemma \ref{lem:home}, respectively.
Find $r_0>0$ such that $r_0\ll\min\{\tilde r,\varepsilon\}$ and $r_0<R_0/2$.
Take a good covering $\mathcal U=\{U_\alpha\}_{\alpha=1}^N$ of $E$ such that $\diam U_\alpha<r_0$ for any $\alpha$.
Let $L(\mathcal U)$ denote the Lebesgue number of $\mathcal U$ and define $\rho<L(\mathcal U)/10$.
Finally, choose $r\ll\rho$ so that $B(p,r)\Subset U(E,\rho)$.

Now suppose $\mu\ll r$ is small enough and take a regular fiber $F:=\bar B(\hat p,r)$.
Then the fibration theorem \ref{thm:adm} shows that $B(p,\tilde r)$ admits a deformation retraction onto $F$.
Similarly $U(\hat E,\varepsilon)$ admits a deformation retraction onto $\bar U(\hat E,\rho/2))$.
Let
\[\sigma:(I^i,\partial I^i)\to (E,p),\quad\hat\sigma:(I^i,\partial I^i,J^{i-1})\to(U(\hat E,\rho),F,\hat p),\]
where $I^i=[0,1]^i$ and $J^{i-1}=\partial I^i\setminus I^{i-1}\times\{1\}$ as before.
We say that $\hat\sigma$ is a \textit{lift} of $\sigma$, or $\sigma$ is a \textit{projection} of $\hat\sigma$, if $|\sigma,\theta\hat\sigma|<r_0$.
Then the desired isomorphism $\pi_i(E,p)\cong\pi_i(U(\hat E,\rho),F,\hat p)$ is given by the correspondence $[\sigma]\leftrightarrow[\hat\sigma]$.

\begin{proof}[Existence of $\sigma$]\renewcommand{\qedsymbol}{}
Given $\hat\sigma$, choose a triangulation $T$ of $I^i$ so that $\diam\hat\sigma(\Delta)<L(\mathcal U)/10$ for any simplex $\Delta\subset T$.
For any vertex $v\in T\setminus\partial I^i$, define $\sigma(v)$ as a point on $E$ closest to $\theta(\hat\sigma(v))$ and $\sigma|_{\partial I^i}:\equiv p$.
Since $\rho<L(\mathcal U)/10$, the images of the vertices of any simplex under $\sigma$ are contained in some $U_\alpha$.
Thus we can define $\sigma$ inductively on the skeleta of $T$ so that
\[\sigma(\partial\Delta)\subset U_\alpha\cap E\Longrightarrow\sigma(\Delta)\subset U_\alpha\cap E\]
for any $\Delta\subset T$ by using the contractibility of $\mathcal U$.
Then $\sigma$ is a projection of $\hat\sigma$ since $\diam U_\alpha< r_0$.
\end{proof}

\begin{proof}[Well-definedness of ${[\sigma]}$]\renewcommand{\qedsymbol}{}
Lemma \ref{lem:home} immediately implies that $\hat\sigma\mapsto[\sigma]$ is well-defined.
Furthermore, the same argument as above shows that $[\hat\sigma]\mapsto[\sigma]$ is well-defined.
\end{proof}

\begin{proof}[Existence of $\hat\sigma$]\renewcommand{\qedsymbol}{}
Given $\sigma$, define $\hat\sigma$ on $I^{i-1}\times\{0\}\cup\partial I^{i-1}\times I$ by the constant map to $\hat p$ and apply Corollary \ref{cor:lifte} to it.
Then $\hat\sigma$ extends over $I^i$ so that $|\sigma,\theta\hat\sigma|<r/2$ since $\mu\ll r$.
In particular $\hat\sigma(I^{i-1}\times\{1\})\subset F$ and $\hat\sigma(I^i)\subset U(\hat E,\rho)$.
\end{proof}

\begin{proof}[Well-definedness of ${[\hat\sigma]}$]\renewcommand{\qedsymbol}{}
Let $H:(I^i,\partial I^i)\times I\to(E,p)$ be a homotopy between $\sigma_0$ and $\sigma_1$, and let $\hat\sigma_0$ and $\hat\sigma_1$ be arbitrary lifts of them.
Define $\hat H:\partial I^{i+1}\setminus(I^{i-1}\times\{1\}\times I)\to M$ by
\[\hat H|_{I^i\times\{0\}}:=\hat\sigma_0,\quad\hat H|_{I^i\times\{1\}}:=\hat\sigma_1,\quad\hat H|_{J^{i-1}\times I}:\equiv\hat p.\]
By Corollary \ref{cor:lifte}, $\hat H$ extends over $I^{i+1}$ so that $|H,\theta\hat H|<\min\{\tilde r,\varepsilon\}/2$ since $r_0\ll\min\{\tilde r,\varepsilon\}$.
In particular $\hat H(I^{i-1}\times\{1\}\times I)\subset B(\hat p,\tilde r)$ and  $\hat H(I^{i+1})\subset U(\hat E,\varepsilon)$.
Using the deformation retractions mentioned above, one can deform it so that $\hat H(I^{i-1}\times\{1\}\times I)\subset F$ and $\hat H(I^{i+1})\subset U(\hat E,\rho)$.
\end{proof}

Therefore the map $[\sigma]\mapsto[\hat\sigma]$ is a well-defined isomorphism.
This completes the proof of Theorem \ref{thm:ext}(2).
\end{proof}


\begin{thebibliography}{99}
\bibitem{A}
S. Alesker, Some conjectures on intrinsic volumes of Riemannian manifolds and Alexandrov spaces, Arnold Math. J. 4 (2018), 1--17.

\bibitem{BBI}
D. Burago, Y. Burago, and S. Ivanov, A course in metric geometry, Grad. Stud. Math., 33, Amer. Math. Soc., Providence, RI, 2001.

\bibitem{BGP}
Yu. Burago, M. Gromov, and G. Perel'man, A.D. Alexandrov spaces with curvature bounded below, Uspekhi Mat. Nauk 47 (1992), no. 2(284), 3--51, 222; translation in Russian Math. Surveys 47 (1992), no. 2, 1--58.

\bibitem{BT}
R. Bott and L. W. Tu, Differential forms in algebraic topology, Grad. Texts in Math., 82, Springer, New York, 1982.

\bibitem{F:reg}
T. Fujioka, Regular points of extremal subsets in Alexandrov spaces, to appear in J. Soc. Math. Japan, \href{https://arxiv.org/abs/1905.05480}{arXiv:1905.05480}, 2019.

\bibitem{F:fibr}
T. Fujioka, A fibration theorem for collapsing sequences of Alexandrov spaces, to appear in J. Topol. Anal., \href{https://doi.org/10.1142/S179352532150028X}{DOI:10.1142/S179352532150028X}, 2021.

\bibitem{K:reg}
V. Kapovitch, Regularity of limits of noncollapsing sequences of manifolds, Geom. Funct. Anal. 12 (2002), 121--137.

\bibitem{K:rest}
V. Kapovitch, Restrictions on collapsing with a lower sectional curvature bound, Math. Z. 249 (2005), 519--539.

\bibitem{K:stab}
V. Kapovitch, Perelman's stability theorem, Metric and comparison geometry, 103--136, Surv. Differ. Geom., 11, Int. Press, Somerville, MA, 2007.

\bibitem{Kw}
K. W. Kwun, Uniqueness of the open cone neighborhood, Proc. Amer. Math. Soc. 15 (1964), no. 3, 476--479.

\bibitem{MY:3dim}
A. Mitsuishi and T. Yamaguchi, Collapsing three-dimensional closed Alexandrov spaces with a lower curvature bound, Trans. Amer. Math. Soc. 367 (2015), 2339--2410.

\bibitem{MY:good}
A. Mitsuishi and T. Yamaguchi, Good coverings of Alexandrov spaces, Trans. Amer. Math. Soc. 372 (2019), 8107--8130.

\bibitem{MY:lip}
A. Mitsuishi and T. Yamaguchi, Lipschitz homotopy convergence of Alexandrov spaces, J. Geom. Anal. 29 (2019), no. 3 2217--2241.

\bibitem{Per:alex}
G. Perelman, Alexandrov's spaces with curvatures bounded from below II, preprint, 1991.

\bibitem{Per:mor}
G. Ya. Perel'man, Elements of Morse theory on Aleksandrov spaces, Algebra i Analiz 5 (1993), no. 1, 232--241; translation in St. Petersburg Math. J. 5 (1994), no. 1, 205--213.

\bibitem{Per:dc}
G. Perelman, DC structure on Alexandrov space, preprint, 1994.

\bibitem{Per:col}
G. Perelman, Collapsing with no proper extremal subsets, Comparison geometry, 149--154, Math. Sci. Res. Inst. Publ., 30, Cambridge Univ. Press, Cambridge, 1997.

\bibitem{PP:ext}
G. Ya. Perel'man and A. M. Petrunin, Extremal subsets in Aleksandrov spaces and the generalized Liberman theorem, Algebra i Analiz 5 (1993), no. 1, 242--256; translation in St. Petersburg Math. J. 5 (1994), no. 1, 215--227.

\bibitem{Pet:para}
A. Petrunin, Parallel transportation for Alexandrov space with curvature bounded below, Geom. Funct. Anal. 8 (1998), 123--148.

\bibitem{Pet:semi}
A. Petrunin, Semiconcave functions in Alexandrov's geometry, Metric and comparison geometry, 137--201, Surv. Differ. Geom., 11, Int. Press, Somerville, MA, 2007.

\bibitem{RX}
X. Rong and S. Xu, Stability of $e^\epsilon$-Lipschitz and co-Lipschitz maps in Gromov-Hausdorff topology, Adv. Math. 231 (2012), 774--797.

\bibitem{SY:3dim}
T. Shioya and T. Yamaguchi, Collapsing three-manifolds under a lower curvature bound, J. Differential Geom. 56 (2000), 1--66.

\bibitem{SY:vol}
T. Shioya and T. Yamaguchi, Volume collapsed three-manifolds with a lower curvature bound, Math. Ann. 333 (2005), 131--155.

\bibitem{Si}
L. C. Siebenmann, Deformation of homeomorphisms on stratified sets, Comment. Math. Helv. 47 (1972), 123--163.

\bibitem{X}
S. Xu, Homotopy lifting property of an $e^\epsilon$-Lipschitz and co-Lipschitz map, preprint, \href{https://arxiv.org/abs/1211.5919}{arXiv:1211.5919}, 2012.

\bibitem{XY}
S. Xu and X. Yao, Margulis lemma and Hurewicz fibration theorem on Alexandrov spaces, to appear in Commun. Contemp. Math., \href{https://doi.org/10.1142/S0219199721500486}{DOI:10.1142/S0219199721500486}, 2021.

\bibitem{Y:col}
T. Yamaguchi, Collapsing and pinching under a lower curvature bound, Ann. of Math. 133 (1991), no. 2, 317--357.

\bibitem{Y:conv}
T. Yamaguchi, A convergence theorem in the geometry of Alexandrov spaces, Actes de la table ronde de g\'eom\'etrie diff\'erentielle, 601--642, S\'emin. Congr., 1, Soc. Math. France, Paris, 1996.

\bibitem{Y:4dim}
T. Yamaguchi, Collapsing 4-manifolds under a lower curvature bound, preprint, \href{https://arxiv.org/abs/1205.0323}{arXiv:1205.0323}, 2012.

\end{thebibliography}
\end{document}